\documentclass[a4paper,12pt,reqno]{amsart}
\usepackage{a4wide}

\usepackage{fourier}
\usepackage{amsmath}
\usepackage{amsfonts}
\usepackage{amsthm}
\usepackage{cite}
\usepackage{hyperref}
\usepackage{subfig}
\usepackage[usenames]{color}

\theoremstyle{plain}% Theorem-like structures provided by amsthm.sty
\newtheorem{theorem}{Theorem}[section]
\newtheorem{lemma}[theorem]{Lemma}
\newtheorem{corollary}[theorem]{Corollary}
\newtheorem{proposition}[theorem]{Proposition}

\theoremstyle{definition}
\newtheorem{definition}[theorem]{Definition}

\theoremstyle{remark}
\newtheorem{remark}{Remark}

\usepackage[utf8]{inputenc}
\usepackage{amsmath,amsfonts,amsthm,amssymb}
\usepackage{xcolor}
\usepackage{ulem}
\usepackage{enumerate}
\usepackage{comment}
\usepackage{algorithm,algcompatible,amsmath}
\usepackage{caption}
\usepackage{graphicx}
\graphicspath{{./images/}}
\numberwithin{equation}{section}
\usepackage{geometry}
\usepackage{hyperref}

\newcommand{\NN}{\mathbb{N}}

\newcommand{\R}{\mathbb{R}}
\newcommand{\norm}[1]{\|{#1}\|}
\newcommand{\abs}[1]{|{#1}|}
\newcommand{\CI}{\mathrm{CI}}

\newcommand{\scp}[2]{\langle {#1},{#2}\rangle}

\newcommand{\eps}{\varepsilon}

\DeclareMathOperator*{\argmin}{argmin}

\DeclareMathOperator{\maj}{maj}
\DeclareMathOperator{\diag}{diag}

\author[D. Ghilli]{Daria Ghilli$^\S$}
\address{$^\S$ Daria Ghilli: LUISS Università di Roma, Dipartimento di Economia e Finanza, Viale Romania 32, Roma, ITALIA}
\email{dghilli@luiss.it}

\author[D. A. Lorenz]{Dirk A. Lorenz$^*$}
\address{$^*$ Dirk A. Lorenz: 
Institute of Analysis and Algebra, TU Braunschweig, Germany}
\email{d.lorenz@tu-bs.de}

\author[E. Resmerita]{Elena Resmerita$^{\#}$}
\address{$^{\#}$ Elena Resmerita:
Institute of Mathematics,
		University of Klagenfurt, Austria}
\email{elena.resmerita@aau.at}

\title[Nonconvex flexible sparsity regularization: Theory and numerical schemes]{Nonconvex flexible sparsity regularization: Theory and mononotone numerical schemes}

\begin{document}

% ------------------------------------------------------------------------------------------------------------
% Abstract
% ------------------------------------------------------------------------------------------------------------

\begin{abstract}
     Flexible sparsity regularization means stably approximating sparse solutions of operator equations by using coefficient-dependent penalizations. We propose and analyse a general nonconvex  approach in this respect, from both  theoretical and  numerical perspectives. Namely, we show convergence of the regularization method and establish convergence properties of a couple of majorization approaches for the associated nonconvex problems. We also test a monotone algorithm for an academic example where the operator is an $M$ matrix, and on a time-dependent optimal control problem, pointing out the advantages of employing variable penalties over a fixed penalty.
\end{abstract}
\maketitle 

\section{Introduction}

This work concerns   reconstructing sparse solutions of ill-posed problems while encouraging  more flexibility than allowed by classical regularization penalties in $\ell^p$ spaces with $p\in (0,\infty)$. Let us first recall some  literature concerning $\ell^p$ sparse regularizers. The inspiring  paper ~\cite{daubechies2004iteratethresh} pointed out their essential role  in promoting sparsity with respect to orthonormal bases for solving inverse problems, in general. The reader is referred  to
\cite{claerbout1973robust,taylor1979deconvolution,levy1981reconstruction,santosa1986linear,chen1998basispursuit} for previous works on specific  inverse problems  and to \cite{tibshirani1996lasso} for statistical approaches (e.g., LASSO). Several theoretical results regarding convergence and error estimates  can be found in  \cite{lorenz2008reglp,grasmair2008sparseregularization,ramlau2010sparse} for the convex case $p\in[1,2)$, and in  \cite{zarzer2009nonconvextikhonov,grasmair2008pleq1,10,bredies2009nonconvexregularization} for the nonconvex case $p\in(0,1)$.
Moreover, it has been shown numerically that  $\ell^p$ terms with $p \in (0,1)$, promote sparsity better than the $\ell^1$-norm (see e.g. \cite{10}), for instance in feature selection and compressed sensing \cite{11} (allowing a smaller number of measurements) or in total variation-based
image restoration \cite{5} (providing better edge preservation). There is an increasing interest in this topic not only within signal processing and image restoration, but also within fracture mechanics, optimal control and optimization - see the  introduction of \cite{GK1} . %Finally,  nonconvex optimization arise from natural image statistics  and it appears to be more effective (more robust) with respect to heavy-tailed distributed noise \cite{51}.

The expansive literature on $\ell^p$  regularization deals with many other fine and interesting aspects, which are however beyond the scope of this work. 

%The focus here is on the flexibility feature, for which more references will be provided along this study,  and  on numerical approaches in a nonconvex setting.

The focus here is on flexible sparsity variational regularization that  facilitates different penalizations for different coefficients of the regularized solution, and on numerical  approaches for the nonconvex regularized problem.   Actually, by  ``flexible sparsity'' we mean  functionals that
regularize different (groups of) indices with a different sparsity
promoting term (see also ~\cite{lorenz2016flexiblesparse}). In general, it is about partitioning  the index set $\NN$ into
sets $I_{1}$, $I_{2}$,\dots and choosing a specific penalty for $I_k$, e.g., based on a sequence of exponents
$p_{k}>0$. Note that these flexible penalties have been used in various applications  for the convex case - see e.g.~\cite{chaari2009minimization,briceno2011proximal,pustelnik2011parallel}.

There are different ways to form a regularization
functional out of these partitions. 
First, one can set the $\phi$-functional as
  \[
  \phi(u) = \sum_{k\in\NN} \sum_{j\in I_{k}}w_{j}|u_{j}|^{p_{k}}
  \]
  with different weights $w_k$. If we denote   $u_{I} = (u_{j})_{j\in I}$ and let $\norm{u_I}_{p,w} = (\sum_{j\in I}w_{j}|u_{j}|^{p})^{1/p}$ for weights $w_{j}>0$, we can write this as
  \[
  \phi(u) = \sum_{k\in\NN}\norm{u_{I_{k}}}_{p_{k},w}^{p_{k}}.
  \]
  A further generalization is considering different functions
  $\phi_{k}:[0,\infty)\to [0,\infty)$ which are nondecreasing and satisfy $\phi_{k}(0)=0$, yielding
  \begin{equation}\label{eq:phi-func-phi}
    \phi(u) = \sum_{k\in\NN} \sum_{j\in I_{k}}\phi_{k}(u_{j}).
  \end{equation}
  Here  explicit weights $w_{k}$ are not needed, as they can be incorporated  in the functions $\phi_{k}$.

Second,  the mixed norms approach  adds the different $p$-norms directly and yields
\begin{equation}\label{eq:mixed-norm}
  \norm{u}_{p_{k}} = \sum_{k\in\NN}\norm{u_{I_{k}}}_{p_{k},w}.
\end{equation}
A notable difference to the $\phi$-functional is that the latter is actually a norm if $p_{k}\geq 1$.
The reader is referred also to the section on group sparsity in  \cite{grasmair2011homogeneous}, where the case $p_k=2$ for each $k$ is tackled.

Further generalizations are possible, e.g., one could also consider a functional of the form
$(\sum_{k\in\NN}\norm{u_{I_{k}}}_{p_{k},w}^{q})^{1/q}$ (i.e. taking
a $q$-norm of the vector of values of the sub-norms instead of the
$1$-norm) or even combine different groups of sub-norms with
different exponents (e.g., one can do this hierarchically). We do not
pursue these generalizations further.

From a theoretical point of view, we extend the work \cite{lorenz2016flexiblesparse} that involves variable exponents penalties. Thus, we are concerned with existence of  minimizers for regularization with  nonconvex functions  $\phi_k$ 
$$
\phi(u)=\sum_{k\in\NN}\phi_k(\abs{u_k}),
$$
 sparsity properties of the minimizers, and convergence to a solution of the original problem in an infinite dimensional setting. As a new challenge, we propose algorithmic approaches for minimizing least squares functions with such general
nonconvex regularization terms, by relying on iterative majorization of the nonconvex part.  Thus, we write each nonconvex  function $\phi_k$ as a minimum of convex functions by means of the concave conjugate of $\phi_k$. Approaches of this type have been (re)discovered several times. To the best of our knowledge, they concern only nonconvex $\ell^p$ penalties with fixed exponent or are discussed only theoretically for variable penalties.  The earliest references we could track down are~\cite{Geman1992,Geman1995,Black1996} - see also \cite{Gorodnitsky1997} using the name iterative reweighting, while more recent rediscoveries are \cite{Chan2014,Lanza2015,37}. We refer also to \cite{Nikolova2005,Allain2006} dealing with convergence of the half-quadratic algorithms.
 There are several possibilities to realize the iterative majorization approach. We state here two methods yielding iterative reweighting algorithms, based  on concave duality for increasing functions. One of them  leads to a sequence of reweighted $\ell^1$-minimization problems (like in~\cite{yu2019iteratively,chen2014convergence}), while the other one leads to a sequence of $\ell^2$-minimization problems, see also~\cite{wipf2010irls}. We investigate convergence of the schemes and  monotonicity in the sense of decreasing the objective
function in every iteration. The $\ell^2$-minimization approach is developed into a full
algorithmic framework in Section \ref{sec:optcond}. %, where  necessary optimality conditions are derived  and solved through a monotonically convergent scheme based on an iterative procedure.
Numerical results in two different situations are given, the first one for an academic M matrix example and the second one for a time-dependent optimal control problem. The scheme is proved to be effective and accurate in both situations. In the optimal control example, we carry out a comparison with the fixed penalty case and emphasize  the advantages of employing
variable penalties.

 A significant reference for the current work is \cite{GK1}, where a monotone scheme for the regularized   problem (with fixed $p$) is introduced. This inspired us to propose the algorithm from Section \ref{sec:optcond}. Moreover, in \cite{GK1}, a primal dual active set strategy is  studied and tested. The adaptation of this strategy  to flexible regularization is out of the scope of the present paper, but will  be a subject  of future work.
 
 One can  mention further numerical schemes  for nonconvex  regularization with a fixed penalty for all coefficients, such as  a generalized gradient projection method \cite{6}, a generalized iterated shrinkage algorithm \cite{52}, and a generalized Krylov subspace method combined with a reweighting technique  \cite{Lanza2015}.
In \cite{44}, a functional approach combined with a gradient technique is proposed, and in \cite{36}  an alternating direction method of multipliers (ADMM) is studied. 
We mention also \cite{XuChXuZh12}, where an iterative half thresholding for fast solution of $L_{1/2}$ regularization is proposed. Finally group sparse optimization problems have been studied in \cite{HuLiMeQiYa17} via an $\ell_{p,q}$ regularization with $p\geq 1$ and $q \in [0,1]$ and numerically solved through a proximal gradient method, and in \cite{BeHa19} through the computation of the proximal mapping and necessary optimality conditions.

Last but not least, a rule on how to choose the flexible functionals is left for future work. We would like to approach this topic  
by learning the penalty.

This paper is structured as follows. In Section \ref{sec:2}, we analyse  the general nonconvex regularization problem.  Section \ref{sec:iterative_majorization}   initiates the discussion on   iterative majorization   approaches for the mentioned problem and presents the necessary concave duality theory. In Section \ref{it_T}  we derive  a simple $\ell^2$-minimization scheme for  flexible sparsity regularization.  Section \ref{sec:optcond} is devoted to the theoretical and numerical  study of  a monotone algorithm that approximates regularized solutions.  A reweighted $\ell^1$-minimization approach is proposed and analysed in Section \ref{1-section}.

\section{Variational regularization with flexible functionals}\label{sec:2}

Let 
\begin{equation}\label{equation}
Au=y
\end{equation}
denote the operator equation we would like to stably solve via variational regularization, where $A:\ell^2\to Y$ is linear and bounded, $Y$ is a Hilbert space.

Consider the following Tikhonov regularization approach
\begin{equation}\label{reg_nonconvex}
    \min_{u\in{\ell^{\phi_k}}}\tfrac12\norm{Au-y}^{2} + \alpha \phi(u),\,\,\,\alpha>0,
\end{equation}
with $\phi$ given by 
\begin{equation}\label{phik}
\phi(u)=\sum_{k\in\NN}\phi_k(\abs{u_k}),
\end{equation}
where
\[
  \ell^{\phi_k} = \{ (u_{k}): \sum_{k}\phi_{k}(|u_{k}|)<\infty\}.
  \]
The following assumption will be used throughout this work:

(A1) The functions  $\phi_{k}:[0,\infty)\to [0,\infty]$  are nondecreasing, continuous, and satisfy $\phi_k(0)=0$ and $\lim_{t\to\infty}\phi_k(t) = \infty$, $\forall k\in\NN$. 

In the sequel, we focus on penalties of the form \eqref{phik}
with all $\phi_k$ simultaneously concave. Several examples of functions $\phi$ in this sense that satisfy (A1) are in order:
\begin{enumerate}
    \item $\displaystyle{\phi(u)= \sum_{k\in\NN} |u_k|^{p_k}}$
    \item $\displaystyle{\phi(u)=\sum_{k\in\NN} \log(|u_k|^{p_k}+1)}$
\end{enumerate}
with $p_k\in (0,1)$.
Inspired  by  \cite{akguen_yildirir} and \cite{Zhao12reweighted-l1-minimization},  the following functions are also of interest:
\begin{enumerate}
    \item[(3)]  $\displaystyle{\phi(u)= \sum_{k\in\NN} \log(|u_k|+1)+ |u_k|^{p_k},\,\,\,p_k\in (0,1)}$
    \item[(4)]  $\displaystyle{\phi(u)= \sum_{k\in\NN} (|u_k|+|u_k|^{p_k})^{q_k},\,\,\,p_k, q_k\in (0,1).}$
    \item[(5)] $\displaystyle{\phi(u)= \sum_{k\in\NN} |u_k|^{p_k}\log(|u_k|+1) ,\,\,\,p_k\in (0,1)}$
    \item[(6)]  $\displaystyle{\phi(u)= \sum_{k\in\NN} |u_k|\left(\log(|u_k|+1)\right)^{p_k} ,\,\,\,p_k\in (0,1)}.$
\end{enumerate}

%somehow similar to the ones in [Bredies and Lorenz, Grasmair] %for the function $\phi$ as in %$\displaystyle{\phi(u)=\sum_{k\in\NN}\phi(|u_k|}$,
 We will show convergence of  minimizers $u_\alpha$ to solutions of \eqref{equation} under additional  conditions on $\phi_k$, that will be verified for most of the examples listed above.
 
 Before that, some theoretical background on  flexible penalties will be described.
 %(A2) $\displaystyle{\lim_{t\to\infty}\phi_k(t)=\infty}.$ ??
 %(A3) There exists a continuous and strictly increasing function
 %$\phi^-:[0,\infty)\to [0,\infty)$ satisfying $\phi^-(0)=0,$ $\lim_{t\to\infty}\phi(t) = \infty$  and 
 %$\phi^-(t)\leq \phi_k(t)$, for all $k\in\NN$ and for all $t>0$. 
 %Question: do we need continuity and monotonicity of $\phi^-$?
%(A4) There exists a continuous and strictly increasing function
 %$\phi^+:[0,\infty)\to [0,\infty)$ vanishing at $0$ and such that
 %$ \phi_k(t)\leq \phi^+(t)$, for all $k\in\NN$ and for all $t>0$. 
%\begin{remark}\label{bound}
%The assumption $\lim_{t\to\infty}\phi(t) = \infty$ implies that %there exists a constant $K>0$ such that $t<K$, whenever $\phi^-(t)<M$ for some $M>0$.
%\end{remark}
%\label{sec:theo-prop-flexible-penalties}
Let us start with a Kadec-Klee (or Radon-Riesz) property   for general functions \eqref{phik}.

\begin{proposition}\label{kadec_klee}
Let $\phi_{k}:[0,\infty)\to [0,\infty)$ satisfy (A1). If $\{u^n\}\subset \ell^{\phi_k}$ converges componentwise to $u\in\ell^{\phi_k}$ and  converges also in the sense $\phi(u^n)\to \phi(u)$ as $n\to\infty$, then the following convergence holds: $\phi(u^n-u)\to 0$ as $n\to\infty$.

\begin{proof} The statement can be proven using  Fatou's Lemma and taking into account that the functions $\phi_k$ are nonnegative   - see the the similar proof of Lemma 2 in \cite{grasmair2008sparseregularization}.

\end{proof}
\end{proposition}

 %(A4) There exists $c>0$ such that 
 %$\displaystyle{\phi_k(|u_k|)\geq \frac{c|u_k|}{1+|u_k|}},$ %$\forall k\in \NN$ (compare to condition (C3') from %\cite{grasmair2010nonconvexsparse}).
 
%\begin{remark}
%	It would be interesting to derive conditions on $\phi_k$ (or %$\varphi$) such that  $\ell^{\phi_k}=\ell^1.$\\\\
%\end{remark}
	The following additional assumptions will be needed in the regularization analysis:
	
	(A2) \[
	\phi_k(t)\geq \frac{ct}{t+1},\,\,\,\forall t>0,\,\,\,\forall k\in\NN,
	\]
	for some $c>0$. This resembles condition (C3') in \cite{grasmair2010nonconvexsparse} imposed to the single function playing there the role of all these $\phi_k$.
	
	(A3) The set $\{t\geq 0: \phi_k(t)\leq M, \,\forall k\in\NN\}$ (where $M>0$) is  bounded in the following sense:
	
	There is $L>0$ (not depending on $k$) such that $0\leq t\leq L$, as soon as   $\phi_k(t)\leq M$, for any $k\in\NN$.
	
	This condition means existence of a uniform bound for the sublevel sets of the functions $\phi_k$.

\begin{remark}\label{examples} Assumptions (A1)-(A3) hold for the following nonconvex functions $\phi_k$ (where (A1) is clearly satisfied for each example):
\begin{enumerate}
    \item $\displaystyle{\phi(u)= \sum_{k\in\NN} |u_k|^{p_k}}$, with $p_k\in (0,1]$, and denote $p=\inf_kp_k>0.$
    Condition (A2) holds with $c=1$, since it amounts to the inequality $t^{p_k}+t^{p_k+1}\geq t$ that is true in both cases $t>1$ and $t\leq 1$. 
    
    Regarding (A3): $t^{p_k}\leq M$ for any $k\in\NN$ means $t\leq M^{1/{p_k}}\leq M^{1/{p}}$ if $M>1$. Hence, $L=\max\{1, M^{1/{p}}\}$.
     %Let $\phi^-(t)=t$ and $\phi^+(t)=t^p$, $\forall t>0,$ since  for any $t\in (0,1)$, one has $t\leq t^{p_k}\leq t^p$. Moreover, (A5) is satisfied with $c=1$.
      \item $\displaystyle{\phi(u)=\sum_{k\in\NN} \log(|u_k|^{p_k}+1)}$.
      
      %Since $\displaystyle{\log(1+t)\geq \frac{t}{1+t}}, \,\forall t>0$, we have
      
      We use the inequality $(a+b)^s\leq a^s+b^s$, for any $a,b>0$ and $s\in [0,1]$. Thus, we have, for any $k\in\NN$ and $t>0$:
      
      $\displaystyle{\log(t^{p_k}+1)\geq \log((t+1)^{p_k})=p_k\log(t+1)\geq p\log(t+1)\geq \frac{pt}{t+1}}$, which yields (A2) with $c=p$. 
      
      Condition (A3) is verified with $L=\max\{1, (e^M-1)^{1/{p}}\}$.
      %Since $\displaystyle{\lim_{t\to 0} \frac{\log(t+1)}{t}}=1$, it follows that $t\mapsto \log(t+1)$ behaves like $t$ when $t$ is small enough, that is, there are constants $c_1,c_2>0$, such that $c_1t<\log(t+1)<c_2t$. Thus, in this case we can consider  $\phi^-(t)=c_1t$ and $\phi^+(t)=c_2t^p$.
\item $\displaystyle{\phi(u)= \sum_{k\in\NN} \log(|u_k|+1)+ |u_k|^{p_k},\,\,\,p_k\in (0,1)}$.
Here one can proceed by combining the previous two examples, thus yielding (A2) with $c=2$ and $L=\max\{e^M, e^M-1+ M^{1/{p}}\}$.
%$\phi^-(t)=c_1t$ and $\phi^+(t)=c_2t^p$.
\item  $\displaystyle{\phi(u)= \sum_{k\in\NN} (|u_k|+|u_k|^{p_k})^{q_k},\,\,\,p_k, q_k\in (0,1).}$
 If $\inf_kq_k=q>0$, then one has
 
 $\displaystyle{ |u_k|^{q_k}\leq(|u_k|+|u_k|^{p_k})^{q_k}\leq |u_k|^{q_k}+|u_k|^{p_kq_k}\leq |u_k|^q+|u_k|^{pq}}$
 due to  the inequality $(a+b)^s\leq a^s+b^s$ for $a,b>0, s\in(0,1)$. Here one can take $c=1$ and $L=\max\{1, M^{1/{q}}\}+\max\{1, M^{1/{pq}}\}$.
 %Here one can take $\phi^-(t)=t$ and $\phi^+(t)=t^q(1+t^{p})$.
 
    \item $\displaystyle{\phi(u)= \sum_{k\in\NN} |u_k|^{p_k}\log(|u_k|+1) ,\,\,\,p_k\in (0,1).}$

    \item[(6)]  $\displaystyle{\phi(u)= \sum_{k\in\NN} |u_k|\left(\log(|u_k|+1)\right)^{p_k} ,\,\,\,p_k\in (0,1)}.$
    
    In the last two examples, the  functions  $\phi_k$ behave like $t\mapsto t^{p_k+1}$, where now $p_k+1>1$. Minimizers of the corresponding Tikhonov functionals do exist (via coercivity with respect to the $\ell^2$ norm), however the sparsity of the minimizers is not anymore guaranteed.
    %The functions $\phi^-(t)=c_1t^2$ and $\phi^+(t)=c_2t^{p+1}$ are the first choices  in the last two examples, since the components of $\phi$ behave like $t\mapsto t^{p+1}$, where now $p+1>1$. Minimizers of the corresponding Tikhonov functionals do exist (via coercivity with respect to the $\ell^2$ norm), however the sparsity of the minimizers is not anymore guaranteed.
\end{enumerate}
Note that the popular nonconvex regularizers SCAD \cite{SCAD} and MCP  \cite{MCP} are not included in the list above since they do not satisfy the limit condition in (A1).

\end{remark}

	The next result will guarantee coercivity of the regularization functional with respect to the $\ell^1$ norm.
	
\begin{lemma}\label{inclusion}
  Let the functions $\phi_k$ verify (A1)-(A3). Then one has $\ell^{\phi_k}\subseteq\ell^1$
  in the following sense:
  
  There exists $C>0$ such that $\displaystyle{C\|u\|_1\leq \sum_k  \phi_k(|u_k|)}$.
\end{lemma}	

\begin{proof}
Let $u\in \ell^{\phi_k}$, that is  $\sum_k\phi_k(|u_k|)<\infty$. Hence, there is $M>0$ such that $\phi_k(|u_k|)\leq M$, for any $k\in\NN$. Condition (A3) implies existence of an $L>0$ satisfying $|u_k|\leq L$, for all $k\in\NN$. By using condition (A2), we have:
\[
\frac{c}{1+L}\|u\|_1\leq \sum_k \frac{c|u_k|}{1+|u_k|}\leq \sum_k  \phi_k(|u_k|).
\]
%Let $u\in \ell^{\phi_k}$. It follows that the series $\sum_k\phi^-(|u_k|\to 0$  as $k\to\infty$. Due to the continuity of the inverse of $\phi^-$ (see the properties of $\phi^-$ in (A3)), one has $|u_k|\to 0$  as $k\to\infty$, which means that $|u_k|$ is sufficiently small for $k$ greater than some $k_0\in\NN$. This implies $\displaystyle{\sum_{k\geq k_0}\phi_k(|u_k|)\geq c\sum_{k\geq k_0}|u_k|},$ which yields the conclusion. 
%is convergent, yielding  $\phi^-(|u_k|)\leq M$, for any $k\in\NN$ and for some $M>0$. Due to Remark \ref{bound}, there is $K>0$ such that $|u_k|\leq K,\,\,\forall k\in\NN$. Since condition (A5) holds, one has
%\[
%\frac{c}{1+K}\|u\|_1\leq \sum_k \frac{c|u_k|}{1+|u_k|}\leq %\sum_k  \phi^-(|u_k|)\leq \sum_k  \phi_k(|u_k|).
%\]
%$\phi^-(|u_k|)\to 0$ as $k\to\infty.$ Due to the continuity  of the inverse of $\phi^-$ (see the properties of $\phi$ in (A3)), one has $|u_k|\to 0$ as $k\to\infty$, which means   there exists $k_0\in\NN$  such that $|u_k|< 1$, for all $k\geq k_0$. This implies $\displaystyle{\sum_{k\geq k_0}\phi_k(|u_k|)\geq c\sum_{k\geq k_0}|u_k|},$ which yields the conclusion. 
\end{proof}
	
%\subsection{Coercivity }
\begin{comment}
Regarding the coercivity of $\phi$: One can  deal with this via
\begin{itemize}
    \item the $\ell^2$ topology.
    \item the $\ell^1$ topology, in case one has $\ell^1=\ell^{\phi_k}$ (ideally!)
\end{itemize}
\end{comment}

%Coercivity of the penalty functionals is now a consequence of (A3), as mentioned below.

\begin{corollary}
  Let the functions $\phi_k$ verify (A1)-(A3). Then $\phi$ is coercive with respect to the $\ell^1$ norm.
\end{corollary}
  %\begin{proof} 
  %and $\phi^-$ satisfy $t\leq \phi^-(t)$ for all $t>0$ sufficiently small. 
 
  %$\phi$ is weakly$^*$ coercive, that is the sublevel sets $\{u\in\ell^{\phi_k}: \phi(u)\leq M\}$ is sequentially weakly$^*$ compact in $\ell^1$, for all $M>0$.
%\end{corollary}

%\end{proof}

%Fix $M>0$ and let  $(u^n)$ be a sequence in $\ell^{\phi_k}$ such that $\phi(u^n)\leq M$, $\forall n\in\NN$. Thus,   using (A3), there exist $c>0$ and $k_0\in\NN$  such that $\|u^n\|_1=\sum_{k\geq k_0}|u_k|+\sum_{k=1}^{k_0-1}}|u_k^n|$ $\phi_k(|u_k^n|)+\geq c |u_k^n|$, for all $k\geq k_0$, for all $n\in\NN$.
%\end{proof}

In the sequel one can see that the minimizers of \eqref{reg_nonconvex} are sparse, whenever they exist.

\begin{proposition}
Let the functions $\phi_k$ satisfy (A1)-(A3). If there is a solution of \eqref{reg_nonconvex}, then it is sparse.
\end{proposition}

\begin{proof}
Let $x$ be a solution of \eqref{reg_nonconvex}. Then one has
\begin{equation}\label{in1}
\frac{1}{2}\|Ax-y\|^2+\alpha \phi(x)\leq \frac{1}{2}\|Au-y\|^2+\alpha \phi(u),
\end{equation}
for any $u\in \ell^{\phi_k}$. 
Note that   $\sum_n\phi_k(|x_k|)<\infty$ implies   $|x_k|\leq L$  for all $k$ and for some $L>0$ (due to (A3)). 
%Note that   $\sum_n\phi^-(|x_n|)<\infty$ implies $|x_n|\to 0$ as $n\to\infty$ and thus, $|x_n|$ is small enough except for a finite number of indices. 
Fix
$i\in\NN$  and define $u^i=x-x_ie_i$, where $e_i$ is the $i$-th vector of the canonical basis.

By using $u=u^i$ in \eqref{in1}, (A2) and the boundedness of $(x_n)$, we have
\[
\frac{c\alpha |x_i|}{L+1}\leq \alpha\phi_i(|x_i|)\leq \frac{x_{i}^2}{2}\|Ae_i\|^2-x_{i}\langle {Ax-y,Ae_i}\rangle.
\]
One can further write
\[
\frac{c\alpha |x_i|}{L+1}\leq \frac{|x_{i}|^2}{2}\|Ae_i\|^2 +|x_{i}|\,|\langle {Ax-y,Ae_i}\rangle|,
\]
which means  that either $x_{i}=0$ or $x_{i}\neq 0$, in which case  one can divide by $|x_{i}|$ and obtain
\[
0<\frac{c\alpha }{L+1}\leq \frac{|x_{i}|}{2}\|Ae_i\|^2 +|\langle {Ax-y,Ae_i}\rangle|=\frac{|x_{i}|}{2}\|Ae_i\|^2 +|\langle {A^*(Ax-y),e_i}\rangle|.
\]
If there are infinitely many $i$ with $x_{i}\neq 0$, then taking $i\to\infty$ above yields a contradiction. This is due to the fact that the right-hand side converges to zero as $\|x\|_1<\infty$ cf. Lemma \ref{inclusion}, and $A^*(Ax-y)=(\langle {A^*(Ax-y),e_i}\rangle)\in\ell^2$ since $A^*:Y\to\ell^2$. Therefore, $x$ has only finitely many nonzero components.
\end{proof}

\begin{remark}
In principle, one could also obtain estimates on the  sparsity level of a solution $x$ of \eqref{reg_nonconvex} in a similar spirit to \cite[Remark 5.1]{bredies2014nonconvexminimization}, however, this would need a known lower bound on the size of the non-zero entries. In \cite{bredies2014nonconvexminimization}, this is known, since the penalty is $\sum_k|x_k|^p$ for some fixed $p<1$. In our case, our assumptions do not allow to conclude the existence of such a lower bound. What we get is, that every minimizer $x$ fulfills
\begin{align*}
    \tfrac12\norm{b}^2 & = \tfrac12\norm{A0-b}^2 + \alpha\phi(0)\\
    & \geq \tfrac12\norm{Ax-b}^2 + \alpha\phi(x)\\
    & \geq \alpha\sum_{x_k\neq 0}\phi_k(|x_k|).
\end{align*}
This implies that
\begin{align*}
    \tfrac12\norm{b}^2\geq \alpha\cdot\inf_{x_k\neq 0}\big[\phi_k(|x_k|)\big]\cdot\#\{k\mid x_k\neq 0\},
\end{align*}
and consequently, the number of nonzero entries in a minimizer $x$ is bounded from above by
\begin{align*}
    \frac{\norm{b}^2}{2\alpha\inf_{x_k\neq 0}\big[\phi_k(|x_k|)\big]},
\end{align*}
if the infimum is non-zero. But this can't be guaranteed by our assumptions.
\end{remark} 

We discuss next existence of solutions of \eqref{reg_nonconvex}.

\begin{proposition} 
  \label{exist_uniq}  Let $A:\ell^2\rightarrow Y$ be a linear and bounded operator which is also $\ell^1$-weak$^*$-weak sequentially continuous, where $Y$ is a Hilbert space. Assume that the functions  $\phi_{k}:[0,\infty)\to [0,\infty)$ verify (A1)-(A3). Then for any $\alpha>0$, the Tikhonov functional \eqref{reg_nonconvex}
has at least one minimizer $u_{\alpha}$ in $\ell^{\phi_k}$.
\end{proposition}

\begin{proof} One can show the statement by usual techniques, taking into account  the coercivity of $\phi$ as above and the weak$^*$ lower semicontinuity of the involved functionals (the latter holds due to the  componentwise convergence of weak$^*$ convergent subsequences and to the lower semicontinuity of each $\phi_k$).
\end{proof}

A convergence result for the proposed regularization method can be formulated in the following.

\begin{proposition} 
 Let $A:\ell^2\rightarrow Y$ be a linear and bounded operator which is also $\ell^1$-weak$^*$-weak sequentially continuous, where $Y$ is a Hilbert space. Assume that there is solution of $Au=y$ and that  the functions $\phi_{k}:[0,\infty)\to [0,\infty)$ satisfy (A1)-(A3).  If $\alpha_n\to 0$ as $n\to\infty$, then there is a subsequence $(u_{\alpha_{n_j}})$ of the  sequence of minimizers $(u_{\alpha_{n}})$ of the problems 
\[
 \min_{u}\tfrac12\norm{Au-y}^{2} + \alpha_n \phi(u),\,\,\,
\]
which converges to a solution $\bar u$ of the operator equation in the sense that $\phi(u_{\alpha_{{n_j}}}-\bar u)\to 0$ as $j\to\infty$ (also in the $\ell^1$-norm).

\end{proposition}\medskip

A similar result holds also in the case that noisy data $y^\delta $  with $\|y-y^\delta\|\leq\delta $ is considered.

\section{Iterative majorization approaches via concave duality}
\label{sec:iterative_majorization}

In the sequel, we focus on numerical  methods for the variational nonconvex problem \eqref{reg_nonconvex}. 
Some notes on existing literature: \cite{wipf2010irls} analyses general reweighting schemes that lead to algorithms similar to the $\ell^{1}$ and $\ell^{2}$ algorithms here, but does not use the majorization approach.  The work \cite{lai2013irls} analyses $\ell^{q}$ minimization with IRLS via $\ell^{1}$ and investigates exact recovery and convergence.

In this section, we consider an approach for minimizing least squares functions with nonconvex regularization terms which is based on iterative majorization of the nonconvex term. The approach has been (re)discovered several times. The first references we know about are \cite{Geman1992,Geman1995,Black1996} and date back to the early 1990s. An apparently independent discovery from the same time can be found in \cite{Gorodnitsky1997} under the name iterative reweighting. Recent rediscoveries are, e.g.  \cite{Chan2014,Lanza2015}, while papers that consider convergence of the respective methods are, e.g., \cite{Nikolova2005,Allain2006}.

The idea to minimize a functional of the form
\[
\frac{1}{2}\norm{Ax-y}^{2} + \alpha\sum_{k=1}^{\infty}\phi_{k}(|x_{k}|)
\]
with functions $\phi_k:[0,\infty)\to[0,\infty)$ which fulfil $\phi_k(0)=0$ is as follows: Since the one-dimensional functions $\phi_{k}$ may not be convex, it is of interest to write each function as a minimum of convex functions.  There are several possibilities to do this, as one can see below.  One can start with using scaled and shifted quadratic functions centered at $0$, i.e., 
\begin{align}\label{eq:def-phik-quad}
    \phi_{k}(t) = \inf_{s\geq 0}(st^{2} + \psi_{k}(s))
\end{align}
for some function $\psi_{k}$.
Alternatively, one can also write $\phi_{k}(t)$ as a minimum of scaled and shifted absolute values centered at $0$,
\begin{align}\label{eq:def-phik}
\phi_{k}(t) = \inf_{s\geq 0}(st + \tilde\psi_{k}(s)).
\end{align}
Figure \ref{fig:majorization-of-sqrt} shows the majorization with absolute values and quadratics for the function $\varphi(t)=\sqrt{\abs{t}}$.

\begin{figure}
    \centering
    \includegraphics{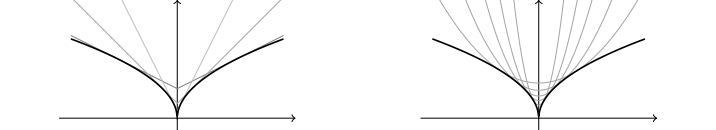}
    \caption{Majorization of the function $\varphi(t) =\sqrt{\abs{t}}$ by shifted and scaled absolute values (left) and quadratics (right).}
    \label{fig:majorization-of-sqrt}
\end{figure}

In the first case~\eqref{eq:def-phik-quad}, the minimization  problem becomes
\[
\min_{x,s\geq 0}\frac{1}{2}\norm{Ax-y}^{2}+ \sum_{k=1}^{\infty}(s_{k}|x_{k}|^{2} + \psi_{k}(s_{k})).
\]
The minimization over $x$ is now a quadratic problem and the minimization over $s$ decouples over $k$, thus yielding simple one-dimensional minimization problems. Note that the minimization over $s$ and $x$ simultaneously is still difficult, but alternating minimization is comparably simpler.
In the second case~\eqref{eq:def-phik},
one would get
\[
\min_{x,s\geq 0}\frac{1}{2}\norm{Ax-y}^{2}+ \sum_{k=1}^{\infty}(s_{k}|x_{k}| + \tilde\psi_{k}(s_{k})).
\]
The minimization over $x$ is an $\ell^1$-penalized least squares problem, for which a number of methods is available \cite{rao1997focuss,daubechies2003iteratethresh,lorenz2008conv_speed_sparsity,lorenz2008ssn,lorenz2008harditer,figueiredo2007gradproj,figueiredo2009sparsa,vandenberg2007spgl1}, while the minimization over $s$, again, decouples over $k$.

Note that  equations~\eqref{eq:def-phik-quad} and~\eqref{eq:def-phik} resemble the definition of the convex conjugate (basically up to sign changes). For example, one can see that $\psi_k$ from~\eqref{eq:def-phik}  fulfills $\psi_k(t) = -(-\phi_k)^*(-t)$, where $(-\phi_k)^*$ is the usual convex conjugate.
Since the multiple minus-signs can be confusing, we present in the following a special notion of duality in the case of concave increasing functions.

%\section{A concave duality for increasing functions on the positive real numbers}
%\label{sec:concave-duality}

  We did not find a reference that contains all the necessary results regarding the concave conjugate and the superdifferential notions for concave functions. Thus,  a presentation of this theory is included for the sake of completeness. It will be largely along the lines of convex duality and hence, most proofs can be omitted. 

We will
investigate the following set of functions:
\[
\CI = \{\psi:[0,\infty)\to[-\infty,\infty)\ \mid \psi\ \text{concave and increasing}\}.
\]
For  $\psi\in \CI$, define
\[
\maj(\psi) = \{(t,l)\in [0,\infty)\times[-\infty,\infty)\ \mid st+l\geq \psi(s),\,\forall s\geq 0\}.
\]
The pairs $(t,l)\in\maj(\psi)$ parametrize the linear functions on $[0,\infty)$ which majorize $\psi$, where  $t$ is its slope and $l$ is the intercept.
Note that this construction is analogous to the situation of the convex conjugate where one shows that a convex and lower semi-continuous function is exactly the supremum of all affine functions below it. This leads to the well known duality theory for convex functions. An adaptation to the situation of concave increasing functions provides us with the following results.

\begin{proposition}
  For any subset $M\subset [0,\infty)\times[-\infty,\infty)$, it holds that the function
  \[
  \varphi(s) = \inf_{(t,l)\in M}st+l
  \]
  defined on $[0,\infty)$ is in $CI$.
  
  Moreover, every function $\psi\in CI$ can be written as
  \[
  \psi(s) = \inf_{(t,l)\in\maj(\psi)}st+l.
  \]
\end{proposition}
\begin{proof}
  Most of the claim follows by analogy to the convex case. The only thing which is missing is the monotonicity, that can be seen as follows: If $s_{1}\leq s_{2}$, then (since $t\geq 0$)
  $s_1t+l\leq s_{2}t+l$ for all $(t,l)\in M$. This shows that
  $\varphi(s_{1})\leq \varphi(s_{2})$.
\end{proof}

\begin{definition}
  For $\psi\in CI$,   the concave conjugate is defined by
  \[
  \psi^{\oplus}(t) = \inf_{s\geq 0}\big(st-\psi(s)\big).
  \]
\end{definition}
It is clear that $\psi^{\oplus}\in CI$.
\begin{proposition}\label{prop:double-conjugate}
  For all $\psi\in CI$, one has  ${\psi^{\oplus}}^{\oplus} = \psi$.
\end{proposition}

We introduce a notion of supergradient for functions in $CI$ (see also  \cite{superdiff}):
\begin{definition}
  For $\psi\in CI$, we say that $s^{*}\geq 0$ is a supergradient of $\psi$ at $s$, if for all $r\geq 0$ it holds that
  \[
  \psi(s) + s^{*}(r-s)\geq \psi(r).
  \]
  We denote the set of all such $s^*$ by ${\hat{\partial}}\psi(s)$ and call it superdifferential.
\end{definition}
Note that ${\hat{\partial}}\psi(s)=-{\partial}(-\psi)(s)$, where $\partial (-\psi)$ is the subdifferential of the convex function $-\psi$.

\begin{theorem}[Fenchel inequality and Fenchel equality]\label{thm:fenchel}
  For $\psi\in CI$, one has
  \[
  \psi(s)+\psi^{\oplus}(s^{*}) \leq ss^{*},\quad \forall s\geq 0, s^*\geq 0.
  \]
  Moreover, $\psi(s)+\psi^{\oplus}(s^{*}) = ss^{*}$ if and only if $s^{*}\in\hat{\partial}\psi(s)$ (or, equivalently $s\in\hat{\partial}\psi^{\oplus}(s^{*})$).
\end{theorem}
\begin{proof}
  The inequality follows from
  \[
  \psi(s)+\psi^{\oplus}(s^{*}) = \psi(s) + \inf_{r\geq 0}(rs^{*}-\psi(r) )\leq \psi(s) + ss^{*}-\psi(s).
  \]
  For the equality, note that $s^{*}\in\hat{\partial}\psi(s)$ is by definition equivalent to  the inequality
  \[
  \psi(s) -ss^{*}\geq \psi(r) - s^{*}r,\ \text{for all $r\geq 0$}.
  \]
  In turn, this is equivalent to
  \[
  \psi(s) -ss^{*}\geq \sup_{r\geq 0}(\psi(r) - s^{*}r) = -\psi^{\oplus}(s^{*})
  \]
  which is the same as $\psi(s)+\psi^{\oplus}(s^{*})\geq ss^{*}$.  The claim follows with Fenchel's inequality.  
\end{proof}

As an example, we take $\psi(s) = s^p$ with $0<p<1$. The concave conjugate is
\begin{align*}
    \psi^\oplus(t) = \inf_{s\geq 0}(st-\psi(s)) = \inf_{s\geq 0}(st-s^p).
\end{align*}
The infimum is attained at $s=(t/p)^{1/(p-1)}$ and the value evaluates to
\begin{align*}
    \psi^\oplus(t) = \Big(\tfrac{1}{p^{1/(p-1)}} - \tfrac{1}{p^{p/(p-1)}}\Big)t^{p/(p-1)}.
\end{align*}
The special case $p=1/2$, for example, leads to $\psi^\oplus(t) = -1/(4t)$ and hence, by Proposition~\ref{prop:double-conjugate} we have the representation
\begin{align*}
    \sqrt{s} = \inf_{t\geq 0}\left(st + \tfrac1{4t}\right).
\end{align*}

%\section{Reweighting algorithms}\label{sec:5}

\section{Iteratively reweighted quadratic regularization}\label{it_T}
This section uses the concave conjugate to derive flexible majorization approaches which lead to simple minimization algorithms for the regularized problems. A noteworthy property  of these methods is their monotonicity, in the sense that they decrease the objective function in every iteration. We consider  iterative  $\ell^2$-minimization (i.e., quadratic Tikhonov regularization) in an infinite dimensional setting and develop the concept into a full algorithmic framework. Moreover, the resulting monotone algorithm will be tested on two examples, showing its efficiency and accuracy.
%in Section~\ref{sec:optcond}.

% As before, let  $A:\ell^2\rightarrow Y$ be a linear and bounded operator which is also $\ell^1$-weak$^*$-weak sequentially continuous, where $Y$ is a Hilbert space. Assume that the functions  $\phi_{k}:[0,\infty)\to [0,\infty)$ verify (A1)-(A3).% Later, we will require also smoothness for the functions $\phi_k$.

To get a series of quadratic problems, we write our objective function as
\[
\min_{x}\norm{Ax-y}^{2}+ \alpha\sum_{k=1}^\infty\psi_k(\abs{x_{k}}^2).
\]
with $\psi_k\in\CI$ for all $k$.
Again, using concave conjugates, we get 
\begin{equation}\label{quadr}
\min_{x\geq 0}\Big[\norm{Ax-y}^{2}+\alpha\sum_{k=1}^\infty s_{k}\abs{x_{k}}^2 - \psi_{k}^{\oplus}(s_{k})\Big].
\end{equation}
If we alternatively minimize for $x$ and $s$, we get
\begin{enumerate}
    \item Initialize with some $x^{0}$ and $s^{0}$, set $n=0$ and iterate
\item Obtain $x^{n+1}$ by solving $A^*(Ax-y)  + \alpha\diag(s^n)x = 0$
\item $s_k^{n+1}\in\hat\partial\psi_k(\abs{x_k^{n+1}}^2)$.
\end{enumerate}

The next lemma shows that the method is indeed of descent type and quantifies the guaranteed descent in each step.

\begin{lemma}\label{lem:descent}
Let $F(x) = \norm{Ax-y}^2 + \alpha\sum_k\psi_k(\abs{x_k}^2)$ and let $x^n$ be the iterates of the above iteration. Then it holds that
\[
F(x^{n+1}) \leq F(x^n) - \norm{Ax^{n+1} - Ax^n}^2 - \alpha\sum_ks_{k}^{n}(x_k^n-x_k^{n+1})^2.
\]
\end{lemma}
\begin{proof}
We write (using $A^*(Ax^{n+1}-y) = -\alpha\diag(s^n)x^{n+1}$)
\begin{align*}
    F(x^n)-F(x^{n+1})& = \norm{Ax^n-y}^2 - \norm{Ax^{n+1}-y}^2 + \alpha\sum_k \psi_k(\abs{x_k^n}^2) -\psi_k(\abs{x_k^{n+1}}^2)\\
        & = 
        \norm{Ax^n-Ax^{n+1}}^2 + \scp{x^n-x^{n+1}}{A^T(Ax^{n+1}-y)} + \alpha \sum_k \psi_k(\abs{x_k^n}^2) -\psi_k(\abs{x_k^{n+1}}^2)\\
        & = \norm{Ax^n-Ax^{n+1}}^2 +  + \alpha \sum_k \psi_k(\abs{x_k^n}^2) -\psi_k(\abs{x_k^{n+1}}^2) - s_k^nx_{k}^{n+1}(x_k^n-x_k^{n+1}).
\end{align*}
Now we use that $\psi_{k}(t^{2})\leq st^{2}-\psi_{k}^{\oplus}(s)$ with $t = x_{k}^{n+1}$ and $s=s_{k}^{n}$ to get
\begin{align*}
  F(x^{n})-F(x^{n+1}) & \geq \norm{Ax^{n}-Ax^{n+1}}^{2} + \alpha\sum_{k}\psi_k(\abs{x_{k}^{n}}^{2}) - s_{k}^{n}(x_{k}^{n+1})^{2} + \psi_{k}^{\oplus}(s_{k}^{n}) - s_{k}^{n}x_{k}^{n+1}(x_{k}^{n}-x_{k}^{n+1}).
\end{align*}
By the Fenchel equality (Theorem ~\ref{thm:fenchel}) and since $s_{k}^{n}\in\hat\partial\psi_{k}(\abs{x_{k}^{n}}^{2})$ we have that $\psi_k(\abs{x_{k}^{n}}^{2}) + \psi_{k}^{\oplus}(s_{k}^{n}) = s_{k}^{n}\abs{x_{k}^{n}}^{2}$ and thus
\begin{align*}
  F(x^{n})-F(x^{n+1}) & \geq \norm{Ax^{n}-Ax^{n+1}}^{2} + \alpha\sum_{k}s_{k}^{n}\abs{x_{k}^{n}}^{2}- s_{k}^{n}\abs{x_{k}^{n+1}}^{2} - s_{k}^{n}x_{k}^{n+1}(x_{k}^{n}-x_{k}^{n+1})\\
  & =  \norm{Ax^{n}-Ax^{n+1}}^{2} + \alpha\sum_{k}s_{k}^{n}(x_{k}^{n}-x_{k}^{n+1})^{2}.
\end{align*}
\end{proof}

\begin{lemma}
  For the iterates $x^{n}$, $s^{n}$, it holds that
  \begin{enumerate}[i)]
  \item $(F(x^{n}))$ is convergent,
  \item  $\displaystyle{\sum_{n=0}^{\infty}\norm{A(x^{n+1}-x^{n})}^{2}\leq F(x^{0})}$,
  \item $\displaystyle{\sum_{n=0}^{\infty}\sum_{k=1}^\infty s^{n}_{k}(x^{n+1}_{k}-x^{n}_{k})^{2}\leq
    F(x^{0})/\alpha}$,
  \item $\displaystyle{\norm{A(x^{n+1}-x^{n})}\to 0}$ for $n\to\infty$, and
  \item  $\displaystyle{\sum_{k=1}^\infty s^{n}_{k}(x^{n+1}_{k}-x^{n}_{k})^{2}\to 0}$ for $n\to\infty$.
  \end{enumerate}
\end{lemma}
\begin{proof}
  Summing up the inequalities from Lemma~\ref{lem:descent} for $n=0$ to $m$, we get
  \[
  F(x^{m+1})\leq F(x^{0}) - \sum_{n=0}^{m}\norm{A(x^{n+1}-x^{n})}^{2} - \alpha\sum_{n=0}^{m}\sum_k s^{n}_{k}(x^{n}_{k}-x^{n+1}_{k})^{2}
  \]
  from which the claim follows.
\end{proof}

The above lemma does not show convergence of the sequence of iterates yet. This could, in principle, be shown by adapting results from \cite{yu2019iteratively}. We do not pursue this here further, but instead, consider the special case of flexible penalties with varying exponents $p_k$ in the next section.

\section{A monotone algorithm for the regularized problem}\label{sec:optcond}

In this section,  the previous approach  by iterative $\ell^2$-regularization will be developed into a full algorithmic framework. For simplicity,  we will focus on the case where all functions $\phi_k$ are of type $\phi_k(t) = t^{p_k}$ and 
will propose an algorithm to solve the  minimization problem
\begin{equation}\label{optprob2m}
\min_{x \in \ell^{p_k}}\frac{1}{2}||A x-y||^2+\alpha \sum_{k\in\NN}|x_k|^{p_k},
\end{equation}
with 
$\displaystyle{\ell^{p_k}:=\{x=(x_k):\,  \sum_{k\in\NN}|x_k|^{p_k}<\infty\}}$ and $(p_k)\subset (0,1)$. The behavior of the algorithm  using other $\phi_k$ functionals will be discussed in the numerics subsection.% Moreover, $(p_k)$ is chosen to satisfy condition $(*)$ from \cite{lorenz2016flexiblesparse} in order to guarantee that $\ell^{p_k}=\ell^1$, that is
\subsection{Convergence properties}

Assume that $\displaystyle{\inf_{k\in\NN}p_k>0}.$%\marginpar{We need this to get coercivity of $J_\eps$. We need also condition (*) from our previous paper.}

We actually consider  
a slightly different version of the quadratic approximation from \eqref{quadr}. That is, we minimize over $x$ the following regularized functional
\begin{equation}
\label{optprobeps2m}
 J_\eps(x)=\frac{1}{2}||Ax-y||^2+\alpha \sum_k \Psi_{\eps, p_k}(|x_k|^2), 
\end{equation}
where for $\eps>0$, $t \geq 0$ and $p \in (0,1)$, %\marginpar{If we define $\Psi$ with $t^2...$ and $t^p...$, then $J_\varepsilon$ uses $\Psi(|x_k|)$ etc, as in the next section. But $\Psi$ will not be concave.}
%\marginpar{DL: I shifted $\Psi_{\eps,p}$ such that $\Psi_{\eps,p}(0)=0$. This should give well definedness and everything else should be ok.} 
\begin{equation}\label{psieps}
\Psi_{\eps,p}(t)= \left\{
\begin{array}{ll}
\frac{p}{2}\frac{t}{\eps^{2-p}} \quad &\mbox{for }\,\, 0\leq t \leq \eps^2\\
\noalign{\smallskip}
t^{\frac{p}{2}}-(1-\frac{p}{2})\eps^p \quad & \mbox{ for }\,\, t \geq \eps^2.
\end{array}
\right.\,
%tyo make bigger: displaystyle
\end{equation}

\begin{figure}
    \centering
    \includegraphics{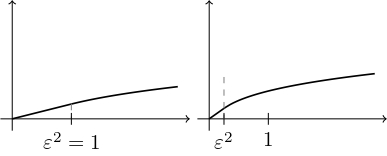}
    \caption{The function $\Psi_{\varepsilon,p}$ from \eqref{psieps} for different values of $\eps$.}
    \label{fig:my_label}
\end{figure}

That is,
\begin{equation}\label{J_e}
J_\eps(x)=\frac{1}{2}||Ax-y||^2+\alpha \sum_{|x_k|\geq \eps}\Big(|x_k|^{p_k} -(1-\tfrac{p_k}2)\eps^{p_k} \Big)+\alpha \sum_{|x_k|\leq \eps}\left( \frac{p_k}{2\eps^{2-p_k}}|x_k|^2\right). 
\end{equation}
Note that $\Psi_{\eps,p}$ is well defined, differentiable, and concave. Clearly, $J_{\varepsilon}$ has minimizers in $\ell^2$.
The necessary optimality condition  is given by
\begin{equation}
\label{optcondeps2m}
A^*Ax+\frac{\alpha p}{\max(\eps^{2-p},|x|^{2-p})}x=A^*y,
\end{equation}
where  the second addend is short for the  sequence with component $\frac{\alpha p_k}{\max(\eps^{2-p_k},|x_k|^{2-p_k})}x_k$.
Note that the optimality condition \eqref{optcondeps2m} identifies local minimizers.

 In order to solve \eqref{optcondeps2m}, the following iterative procedure is considered:
\begin{equation}\label{iter2m}
A^*Ax^{i+1}+\frac{\alpha p}{\max(\eps^{2-p},|x^{i}|^{2-p})}x^{i+1}=A^*y,\,{i\in\NN,}
\end{equation}
where nonlinear functions are applied component-wise in the second addend that has components $\frac{\alpha p_k}{\max(\eps^{2-p_k},|x_k^{i}|^{2-p_k})}x_k^{i+1}$. {Note that \eqref{iter2m} is well defined.}

We have the following convergence result.

\begin{theorem}\label{monotdec2m}
 For $\eps>0$, let $(x^i)$ be generated by \eqref{iter2m}. Then $(J_{\eps}(x^i))$ is strictly monotonically
	decreasing, unless there exists some $i$ such that $x^i = x^{i+1}$, and $x^i$ satisfies the necessary optimality condition \eqref{optcondeps2m}. Moreover, every  cluster point of $(x^i)$, of which there exists at least one, is a solution of \eqref{optcondeps2m}.
\end{theorem}
\begin{proof}
\begin{comment}
Since 
\begin{equation*}
%\tfrac12\tfrac{p_k}{\max(\eps^{2-p_k},\abs{x}^{2-p_k}}\in\hat\partial\Psi_{\eps,p_k}(\abs{x}),
\end{equation*}
%we can apply Lemma~\ref{lem:descent} to obtain monotone decrease of the objective $J_\eps$, i.e. 
%$(J_\eps(x^i))$ decreases and is bounded.
\end{comment}
The proof follows similar arguments to that of Theorem 4.1
in \cite{GK1} and  of Lemma \ref{lem:descent}. For the sake of completeness, we sketch the main proof steps. By applying \eqref{iter2m} to $x^{i+1}-x^i$, we get
\begin{multline*}
\frac{1}{2}\|Ax^{i+1}\|^2-\frac{1}{2}\|Ax^{i}\|^2+\frac{1}{2}\|A(x^{i+1}-x^{i})\|^2+\alpha\left\langle \frac{ p}{\max(\eps^{2-p},|x^{i}|^{2-p})}x^{i+1},x^{i+1}-x^{i}\right\rangle\\
=\langle A^*y,x^{i+1}-x^i\rangle.
\end{multline*}

Note that
\begin{equation}\label{eq111}
\left\langle \frac{ p}{\max(\eps^{2-p},|x^{i}|^{2-p})}x^{i+1},x^{i+1}-x^{i}\right\rangle=\frac{1}{2}\sum_{k\in\NN}\frac{p_k(|x^{i+1}_k|^2-|x^{i}_k|^2+|x^{i+1}_k-x_k^i|^2)}{\max(\eps^{2-p_k},|x^{i}_k|^{2-p_k})}
\end{equation}
and 
\begin{equation}\label{eq112}
    \frac{1}{2} \frac{ p_k}{\max(\eps^{2-p_k},|x^{i}_k|^{2-p_k})}(|x^{i+1}_k|^2-|x^{i}_k|^2)=\Psi_{\eps,p_k}'(|x_k^i|^2)(|x^{i+1}_k|^2-|x^{i}_k|^2).
\end{equation}
Since $\Psi_{\eps,p_k}$ is concave, we have
\begin{equation}\label{eq113}
    \Psi_{\eps,p_k}(|x^{i+1}_k|^2)-\Psi_{\eps,p_k}(|x^{i}_k|^2)- \frac{1}{2} \frac{ p_k}{\max(\eps^{2-p_k},|x^{i}_k|^{2-p_k})}(|x^{i+1}_k|^2-|x^{i}_k|^2)\leq 0.
\end{equation}
Then,  using \eqref{eq111}, \eqref{eq112}, \eqref{eq113}, we get
\begin{equation}\label{eq114}
    J_{\varepsilon}(x^{i+1})+\frac{1}{2}\|A(x^{i+1}-x^{i}\|^2+\frac{1}{2}\sum_{k\in\NN}\frac{\alpha p_k}{\max(\eps^{2-p_k},|x^{i}_k|^{2-p_k})}|x_k^{i+1}-x_k^i|^2\leq J_{\varepsilon}(x^{i}).
\end{equation}
%from $\displaystyle{\sum_{k\in\NN}|x^{i}_k|^{p_k}\leq C}$ for some $C>0$, the fact that $\displaystyle{\inf_{k\in\NN}p_k>0}$ and from Section 4 in \cite{lorenz2016flexiblesparse},
From \eqref{eq114} and coercivity of $J_{\varepsilon}$ it follows that $(x^i_k) $ is bounded in  $\ell^2$ and hence in $\ell^\infty$. This and $\displaystyle{\inf_{k\in\NN}p_k>0}$ imply existence of a constant $\kappa>0$ such that
\begin{equation}\label{45}
J_\eps(x^{i+1}) +\frac{1}{2}||A(x^{i+1}-x^i)||^2+\kappa||x^{i+1}-x^{i}||^2  \leq J_\eps(x^i),
\end{equation}
from which we conclude the first  part of the theorem.
From \eqref{45}, we deduce that
\begin{equation}\label{46}
\sum_i ||A(x^{i+1}-x^i)||^2+\kappa||x^{i+1}-x^{i}||^2  <\infty.
\end{equation}
Since $(x^i)$ is bounded in $\ell^2$, there exists a subsequence $(x^{i_l})$ and $\bar x \in \ell^2$ such that $(x^{i_l})$ converges weakly to $\bar x$ in $\ell^2$. By \eqref{46}  we have that $\displaystyle{\lim_{l\to \infty}x_k^{i_{l+1}}=\lim_{l\to \infty}x_k^{i_{l}}=\bar x_k}$. Testing \eqref{iter2m} with $e_k, k=1, \dots,$ and passing to the limit with respect to $i$ in  \eqref{iter2m}, we get that $\bar x$ is a solution to \eqref{optcondeps2m}.
\end{proof}
\begin{remark}
Theorem \ref{monotdec2m} shows that each cluster point is stationary. The extension to global convergence  under the Kurdyka-Lojasiewicz (KL) property seems a difficult task. We did not succeed to recover such a convergence, since it is not clear how to prove that $\psi  \circ J_\eps$ (where $\psi$ is the KL function) is a Lyapounov function with decreasing rate close to $||x^{i+1}-x^i||$, which is the key point on which the KL argument relies. Take, for example, a (classical) reference for  the proof of the convergence under the KL property: Theorem $8$ of paper \cite{A}. In the following inequality due to the KL property,
$$
\psi'(J_\epsilon(x^i))||\partial J_\epsilon(x^i)||\geq 1,
$$
 one needs to estimate the lefthand side  in terms of $||x^i-x^{i-1}||$.
In our case, we have
$$
\partial J_\epsilon(x^i)=A^*Ax^i-A^*y+\frac{\alpha p}{\max\{\epsilon^{2-p}, |x^i|^{2-p}\}}x^i,
$$
where
$$
A^*Ax^{i}-A^*y=-\frac{\alpha p}{\max\{\epsilon^{2-p}, |x^{i-1}|^{2-p}\}}x^{i}.
$$
Hence,
$$
\partial J_\epsilon(x^i)=\frac{\alpha p}{\max\{\epsilon^{2-p}, |x^i|^{2-p}\}}x^i-\frac{\alpha p}{\max\{\epsilon^{2-p}, |x^{i-1}|^{2-p}\}}x^{i}
$$
should be  estimated in terms of $||x^i-x^{i-1}||$ for all $\epsilon>0$, which may not be possible.\\
Moreover we remark that the KL property has been proved for the $\ell^p, p\in [0,1)$ nonconvex penalty (see \cite{KLellp}, section $5$), but in the case of a flexible $\ell^{p_k}, (p_k) \in (0,1)$ penalty like ours, it is not straightforward (at least one would need some further assumptions on the sequence $p_k$, which we did not investigate in detail).
\end{remark}
\begin{remark}
Note that  the same result holds true in finite dimension for the problem
\begin{equation}\label{optprob2mlambda}
\min_{x \in \R^d}\frac{1}{2}||A x-y||^2+\alpha \sum_k|(\Lambda x)_k|^{p_k},
\end{equation}
where $\Lambda \in \R^{m\times d}$. However, in order to have existence of a solution $\Lambda$ must satisfy some additional assumptions, e.g. $\mbox{Ker}(A) \cap \mbox{Ker}(\Lambda)=\{0\}$. %Moreover, in Section \ref{alg} we will present a second type of scheme, a monotone primal dual active set scheme, which requires $\Lambda$ to be an invertible operator. This seemed to us a too restrictive assumption in the infinite dimensional case. 
The presence of such kind of  operator was thoroughly investigated in \cite{GK1}, where the problem is finite dimensional. One main focus of \cite{GK1} was an application in fracture mechanics, where the presence of the operator $\Lambda$ is crucial. In the present paper we are  mostly interested in the analysis of the $p_k$-sequence approach, whose setting does not seem to provide any new insight into the fracture mechanics examples considered, e.g., in \cite{GK1}. Indeed, in those kind of examples the regularization term is active just on one component of the solution. However, even though  in the present paper we do not focus on the analysis of the case of an operator $\Lambda \neq I$ inside the regularization term,  we remark that  in the numerical tests such kind of operator could be used.
\end{remark}
The following proposition establishes the convergence of minimizers of \eqref{optprobeps2m} to solutions of \eqref{optprob2m} as $\eps$  goes to zero. The proof uses a technique similar  to the one of Proposition 1 of \cite{GK1}.
As it is easily understandable from the above mentioned proof, we just underline that  the following result gives an $\ell^2$-convergence of the sequence $(x_\eps)_{\eps>0}$ as $\eps\to 0$ to any cluster point.
\begin{proposition}
Let $(x_\eps)_{\eps>0}$ be solutions to \eqref{optprobeps2m}. Then any cluster point of  $(x_\eps)_{\eps>0}$ as $\eps \to 0^{+}$, of which there exists at least one,  is a solution of \eqref{optprob2m}.
\end{proposition}

\subsection{Numerical results}\label{numericsDeps}

In the sequel we investigate the performance of the monotone algorithm in practice. For this reason we will work in finite dimension.   Our aims are showing the good performance of the $(p_k)$ sequences approach in the monotone scheme,  and  comparing these results  to the ones of \cite{GK1} with $p$ fixed (and $\Lambda=I)$. Thus, we consider two problems  that have been investigated also in \cite{GK1}: the first one is an academic example where the operator $A$ is an $M$ matrix, and the second one is  a time-dependent optimal control problem. 
We have tested two different kinds of penalty, first $\phi_k(t)=t^{p_k}$, and secondly $\phi_k(t)=\log(t^{p_k}+1)$. We describe our results in detail for the first case, and write some remarks on the second case, to underline the main differences.

Our findings show that the $(p_k)$ sequences approach with $p_k \in (0,1)$ proves to be at least as effective as the $p$ fixed one (with $p \in (0,1)$), and in some cases, even more effective. Indeed,  the $(p_k)$ sequence approach is more flexible in the sense that it enables to concentrate the "sparsity" constraint where it is more needed. We have exploited this feature in the control example in the way we have chosen the sequence $(p_k)$.
In this example, we are able to localise easily, a priori, an approximate region where the solution is known to be zero. Therefore we "relax" the sparsity constraint by choosing sequences $(p_k)$ which are closer to $1$ in those regions.  In the areas where the solution is expected to be nonzero, we let the sequence tend towards the value used in \cite{GK1} in the same situation. 
Choosing sequences as described above simplifies the scheme's performance. Consistently with our expectations, the $(p_k)$ sequence approach shows a smaller number of iterations and a lower residue as compared to the approach of \cite{GK1} with $p$ fixed (see  Table \ref{tablecontrol1} and Table \ref{tablecontrol2}). 
Moreover, the $(p_k)$ sequences approach  provides better sparsity as compared to the approach of \cite{GK1} with $p$ fixed (see  Table \ref{tablecontrol1} and Table \ref{tablecontrol2}).

The situation is a bit different in the academic $M$ matrix example, since we are not able to easily know a priori information on the behaviour of the solution. However, when testing different types of $(p_k) \subset (0,1)$ sequences, we have observed more sparsity for sequences  $p_k \to p$ where $p$ is the value used in \cite{GK1}. For this reason, we will choose this type of sequences when comparing with the results of \cite{GK1}. 
The aim of this experiment is to show the good performance of the $(p_k)$ sequence approach in an academic example (with a  residue  always $O(10^{-9})$), and to show the different behaviour of the solution when changing the sequence $(p_k)\subset (0,1)$.

%\subsection{The numerical scheme}
For convenience of the exposition, we write the algorithm in the following form, see \textbf{Algorithm $1$} for $\phi_k(t)=t^{p_k}$.
A  continuation strategy with respect  to the parameter $\eps$ is performed. In each of the following example, the initialization and range of $\eps$-values will be described.
 The key modifications in the case $\phi_k(t)=\log(t^{p_k}+1)$ are described in Remark \ref{rem:algphilog}.

The algorithm stops when the $\ell^\infty$-norm of the residue of \eqref{optcondeps2m} is $O(10^{-8})$ in the M Matrix example and  $O(10^{-15})$ in the time control problem example. At this instance, the $\ell^2$-residue is typically much smaller. Therefore, we find an approximate solution of the $\eps$-regularized optimality condition \eqref{optcondeps2m}.
The system  in \eqref{systemalg} is solved through the MATLAB function \textit{mldivive} (that is, the \textit{backslash} command).
The  initialization $x^0$  is chosen as the solution of the problem \eqref{optprob2m} where the $\ell^{p_k}$-term is replaced by the $\ell^2$-norm, that is,
\begin{equation}\label{initmoneps}
x^0=(A^*A+2\alpha)^{-1}A^*y.
\end{equation}
Inspired by the findings of \cite{GK1}, we have observed that for some values of $\alpha$ the previous initialization is not suitable (that is, the obtained residue  is too big).  In order to overcome the problem, we have successfully tested a continuation strategy with respect to increasing $\alpha$-values.
\begin{algorithm}[ht]
	\caption{Monotone algorithm + $\eps$-continuation strategy}
	\begin{algorithmic}[1]
		\STATE Initialize $\varepsilon=\varepsilon^0$, $x^0$ and $i=0$;
		\REPEAT
		\STATE  Solve for $x^{i+1}$
		\begin{equation}\label{systemalg}
		A^*Ax^{i+1}+\frac{\alpha p}{\max(\eps^{2-p},|x^{i}|^{2-p})} x^{i+1}=A^*y,
		\end{equation}
		where  the second addend is short for the vectors with component $\frac{\alpha p}{\max(\eps^{2-p_k},|x^i_k|^{2-p_k})}x^{i+1}_k$.
		\STATE Set $i=i+1$.
		\UNTIL{the stopping criterion is fulfilled}.
		\STATE  Reduce $\eps$ and repeat 2.
	\end{algorithmic}
\end{algorithm}

\begin{remark}\label{rem:algphilog}
When $\phi_k(t)=\log(t^{p_k}+1)$, the regularization \eqref{optprobeps2m} becomes
$$
 J_\eps(x)=\frac{1}{2}||Ax-y||^2+\alpha \sum_{k\in \mathbb{N}} \log(\Psi_{\eps, p_k}(|x_k|^2)+1),
$$
and equation \eqref{systemalg} written in the following compact form becomes
$$
	(A^*A+FM)x^{i+1}=A^*y,
$$
where $F$ is the diagonal matrix with $k$ component $\frac{1}{\Psi_{\eps, p_k}(|x^i_k|^2)+1}$ and $M$ is the diagonal matrix with $k$-component $\frac{\alpha p_k}{\max(\eps^{2-p_k},|x_k^{i}|^{2-p_k})}$.
\end{remark}

Note that, in the exposition of the numerical results, the total number of iterations shown in the tables  takes into account the continuation strategy with respect to $\eps$, but  it does not take into account the continuation with respect to $\alpha$.
Finally we underline that in all the tests described in the following,  the value of the objective functional for each iteration was tested to be monotonically decreasing. Note that this is consistent with the result of  Theorem \ref{monotdec2m}.

In the sequel we will adopt the following notation.  For $x \in \R^d$ we will denote  $|x|_0=\#\{k\, :\, |x_k|> 10^{-10}\},$  $|x|_0^c=\#\{k\, :\, |x_k|\leq 10^{-10}\},$ and by $||x||$  the euclidean norm of $x$.

%This indicates that the components were the $\eps$-regularization is most influential are very close to zero.
%In the following  section we discuss an algorithm who aims at solving the original unregularized problem.\\

\subsubsection{M-Matrix example}
Here  the monotone scheme will be tested for the M-matrix example. This is a classical academic example that we have chosen to confirm the good performance of our scheme for different values of the sequence $(p_k) \subset (0,1]$. More specifically, we consider
\begin{equation}
\label{optprobM2M}
\min_{x \in \R^{d^2}}\frac{1}{2}||A x-y||^2+\alpha \sum_{k=1}^{d^2}|x_k|^{p_k},
\end{equation}
where $A$ is the backward finite difference gradient
\begin{equation}\label{A}
A=(d+1)\left(\begin{array}{c} G_1\\G_2\end{array}\right),
\end{equation}
with $G_1 \in \R^{d(d+1)\times d^2}, G_2 \in \R^{d(d+1)\times d^2}$ given by
$$
G_1=I \otimes D, \quad G_2=D \otimes I.
$$
Here $I$ is the $d\times d$ identity matrix, $\otimes$ denotes the tensor product and  $D\in \R^{(d+1)\times d}$ is given by
\begin{equation}\label{D}
D=\left(\begin{array}{ccccc}
1& 0& 0& \cdots& 0\\ -1& 1& 0& \cdots& 0\\ \vspace{0.2cm}\\ 0& \cdots& 0&-1&1\\0&\cdots&0&0&-1
\end{array}\right).
\end{equation}
Then $A^* A$ is an $M$ matrix coinciding with the $5$-point star discretization on a uniform mesh on a square of the Laplacian with Dirichlet boundary conditions.  Note that \eqref{optprobM2M} can be equivalently expressed as
\begin{equation}
\label{optprob2}
\min_{x \in \R^{d^2}}\frac{1}{2}\|A x\|^2-\langle x,f\rangle +\alpha \sum_{k=1}^{d^2}|x_k|^{p_k},
\end{equation}
where $f=A^* y$. If $\alpha=0$, this is the discretized variational form of the elliptic equation
\begin{equation}
\label{elleq}
-\Delta z=f \mbox{ in } \Omega, \quad z=0 \mbox{ on } \partial \Omega.
\end{equation}
For $\alpha>0$, the variational problem \eqref{optprob2} gives a solution with piecewise constant enhancing behaviour.

We choose $f$  as in \cite{GK1}, subsection $3.4$, namely $f$ is a discretization of 

$f(x_1,x_2)=10 x_1\mbox{sin}(5x_2) \mbox{cos}(7 x_1)$.
%In our tests  $p=.1$ and $\beta$ takes increasing values from $10^{-4}$ to $10$.
The parameter $\eps$ varies in the same range as in \cite{GK1}, that is, it was initialized with $10^{-1}$ and decreased  to $10^{-6}$.

In Table \ref{tableMMatrix} we show the performance of \textbf{Algorithm $1$} for  $p\to 0.1$,  $h=1/64$ as mesh size and  $\alpha$ incrementally increasing by factor of $10$ from $10^{-4}$ to $10$.  The sequence $(p_k)$ is approximated by a vector of $N=63^2$ points  with  $p_0=1.1$, generated by the MATLAB command $(p_k)=0.1+1./P$ with $P=linspace(1, 100, N)$.
In Figure \ref{fig:MMatrix} we show the graphics of the solutions for different values of $\alpha$ for which  most changes occur. 

One can observe significant differences  with respect to different values of $\alpha$. The third row of Table \ref{fig:MMatrix} shows that $|x|^c_0$ increases with $\alpha$, as expected. For example, for $\alpha=1, 10$, we have $|x|^c_0=3969$, or equivalently, $|x|_0=0$, that is, the solution to \eqref{optprob2} is  constantly zero. Moreover, we see that  $| x|^p_p$  decreases when $\alpha$ increases (see the fourth row).
In the fifth row we show the $\ell^\infty$ norm of the residue, which is $O(10^{-9})$ for  all the considered $\alpha$.

Note also that  the number of iterations is sensitive with respect to $\alpha$, in particular it increases for $\alpha$  increasing from $10^{-4}$ to $10^{-1}$ and  it decreases consistently for $\alpha=1,10$.

The algorithm  was also tested for different values of $(p_k) \subset (0,1)$. The results obtained show dependence on $(p_k)$. In particular in Figure \ref{fig:MMatrix2} we show the graphics of the solution for $(p_k)$ a random sequence, that is a vector of $N=63^2$ uniformly distributed random numbers in the interval $(0,1)$, obtained by the command \textit{rand} in MATLAB. Comparing Figure \ref{fig:MMatrix2} B) to Figure  \ref{fig:MMatrix} B), we see that  for $\alpha=0.1$  the solution is less sparse as compared to the solution for $p_k \to 0.1$.

For any $k$, we say that $k$ is a  singular component of the vector $ x$ if $k \in\{k \, :\, |x_k|<\eps\}$. In particular, note that the singular components are the ones where the $\eps$-regularization is most influential. In the sixth row of the tables  we show $Sp$, which denotes the number of the singular components of $x$. Note that it coincides with the quantity $|x|_0^c$, which reassures the validity of the $\eps$-strategy.

%In particular, for $\beta=.3$ we have $|D_1x|_0^c=3636, Sp_1=3640$ and $|D_2x|_0^c=3612, Sp_2=3616$. Also, for $\beta=1,10$ we have $Sp_1=Sp_2=|D_1x|_0^c=|D_2x|_0^c=3969$. This indicates that the components were the $\eps$-regularization is most influential are very close to zero.\\
%We refer to subsection \ref{concl} for further remarks about singular points.

Moreover, we tested sequences $p_k \to 0$, but no significant improvement was remarked with respect to the above explained results. In part, this seems coherent with the results of \cite{GK1} where $p=0$ was not tested successfully.
%A more suitable scheme to test sequences $p_k\to 0$ seems the primal dual active set strategy proposed in subsection \ref{numericsD} and  which was previously tested in \cite{KI} for $p=0$ (p fixed).

We tested the algorithm when $\phi_k(t)=\log(t^{p_k}+1)$, according to the modification shown in Remark \ref{rem:algphilog}. The results are shown in Figure \ref{fig:MMatrixphi} only in the case $\alpha=0.1$ and $\alpha=1$, since no significant changes are observed for smaller values of $\alpha$ with respect to $\alpha=0.1$. Comparing Figure \ref{fig:MMatrixphi} a) with Figure \ref{fig:MMatrix} b) and Figure \ref{fig:MMatrixphi} b) with Figure \ref{fig:MMatrix} c), the solution is less sparse. 
Moreover, the number of iterations and the residue are much higher with respect to $\phi_k(t)=t^{p_k}$, respectively of the order $10^3$ and $10^{-1}$. 

Finally, we remark that if we modify the initialization \eqref{initmoneps}, the method converges to the same solution with no remarkable modifications in the number of iterations, which is a sign for the global nature of the algorithm.

\begin{table}[ht]
	\captionof{table}{$M$-matrix example, $p_k \to0.1$, mesh size $h=\frac{1}{64}$. Results obtained by \textbf{Algorithm $1$}.} %, res=$10^{-8}$}
	\centering
	%$eps=10^{-4}:10^{-8}$, zero=$10^{-12}$} \label{tab:title}
	\begin{tabular}{|l|c|c|c|c|c|c|c|}
		\hline\noalign{\smallskip}
		{\bf $\alpha$ }   &$10^{-4}$& $10^{-3}$ &$10^{-2}$&$10^{-1}$ & $1$&10 \\		
		\noalign{\smallskip}\hline\noalign{\smallskip}
		no. of iterates  & 163&785& 2046&1984& 131& 15 \\
		{\bf $|x|^c_0$ } &7 &41& 269&2617& 3969&3969\\
		{\bf $| x|^p_p$ }& $2*10^3$&$2*10^3$&$ 2* 10^{3}$&$812$&9.6&7.3\\
		$\mbox{ Residue }$ & $8*10^{-10}$ & $8*10^{-9}$ & $ 9*10^{-9}$ &  $4*10^{-9}$  & $3*10^{-15}$&$10^{-15}$ \\
		$Sp$  & 7&41&269& 2617& 3869&3969\\
		\noalign{\smallskip}\hline
		%$\mbox{inf norm sing. points}$  & $0$&$6*10^{-12}$&$9*10^{-15}$&$7*10^{-17}$&$7*10^{-18}$ & $7*10^{-19}$\\
		%\hline
	\end{tabular}
\label{tableMMatrix}
\end{table}

\begin{figure}[ht]
	\centering

	\subfloat[$\alpha=0.01$]
	{
		\includegraphics[height=5cm, width=4.5cm]{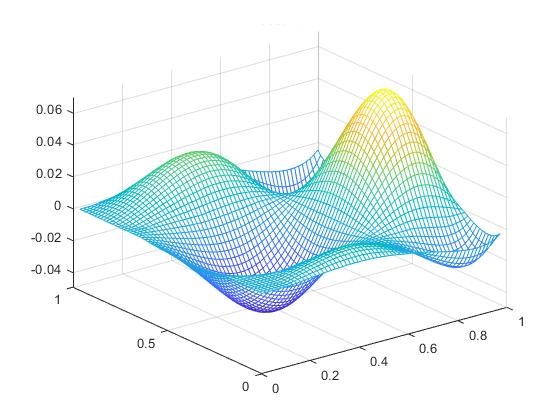}

}	
	\subfloat[$\alpha=0.1$]
	{
		\includegraphics[height=5cm, width=4.5cm]{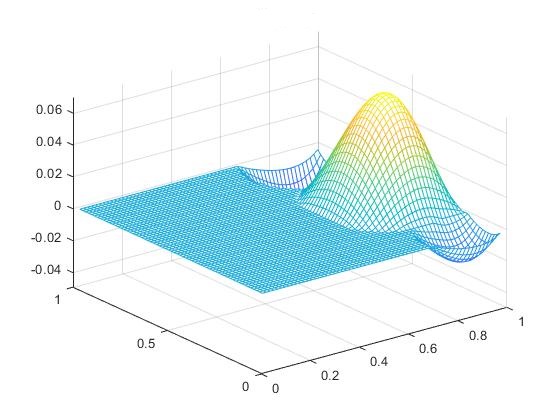}
	}
	\subfloat[$\alpha=1$]
	{
		\includegraphics[height=5cm, width=4.5cm]{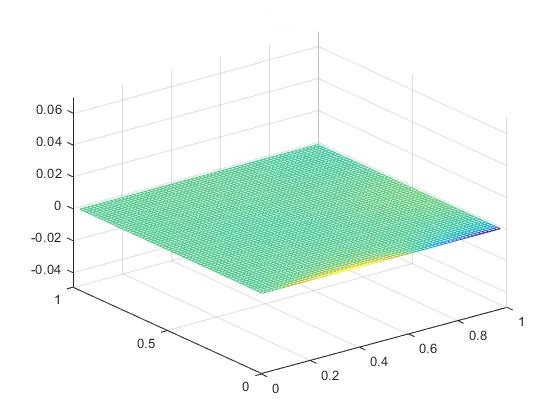}
	}
	%\end{figure}
	%\begin{figure}[h!]
	%\centering
	%\subfigure
	%{
	%\includegraphics[height=5cm, width=4.5cm]{epsM7.png}
	%}
	\caption{Solution of the M-matrix problem, $p_k\to^+ 0.1,$ mesh size $h=\frac{1}{64}$. Results obtained by \textbf{Algorithm $1$}.}
	\label{fig:MMatrix}
\end{figure}

\begin{figure}[ht]
	\centering
	
	\subfloat[$\alpha=0.01$]
	{
		\includegraphics[height=5cm, width=4.5cm]{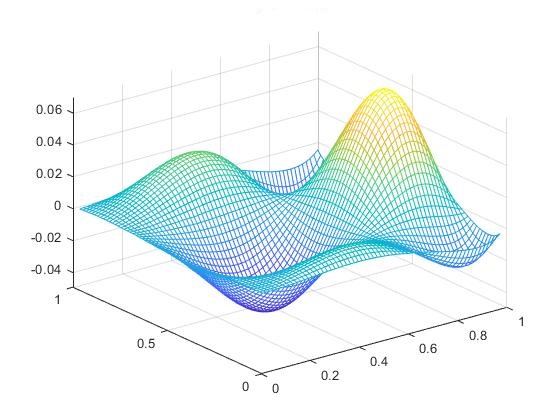}
		
	}	
	\subfloat[$\alpha=0.1$]
	{
		\includegraphics[height=5cm, width=4.5cm]{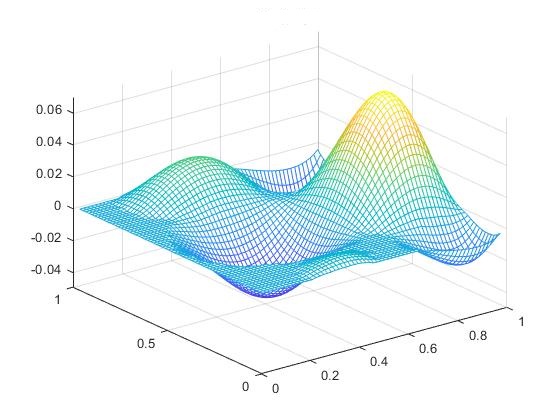}
	}
	\subfloat[$\alpha=1$]
	{
		\includegraphics[height=5cm, width=4.5cm]{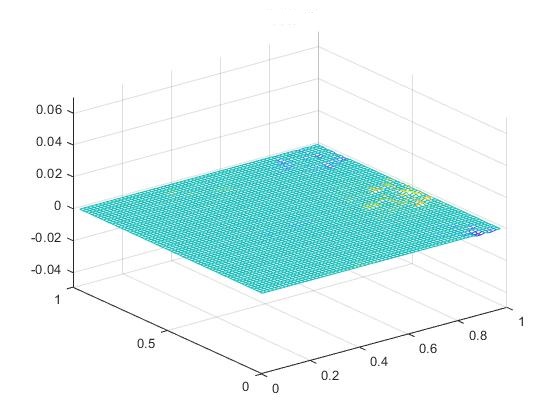}
	}
	%\end{figure}
	%\begin{figure}[h!]
	%\centering
	%\subfigure
	%{
	%\includegraphics[height=5cm, width=4.5cm]{epsM7.png}
	%}
	\caption{Solution of the M-matrix problem, $p_k \in (0,1]$ random sequence, mesh size $h=\frac{1}{64}$. Results obtained by \textbf{Algorithm $1$}.}
	\label{fig:MMatrix2}
\end{figure}

\begin{figure}[ht]
	\centering

	\subfloat[$\alpha=0.1$]
	{
		\includegraphics[height=5cm, width=4.5cm]{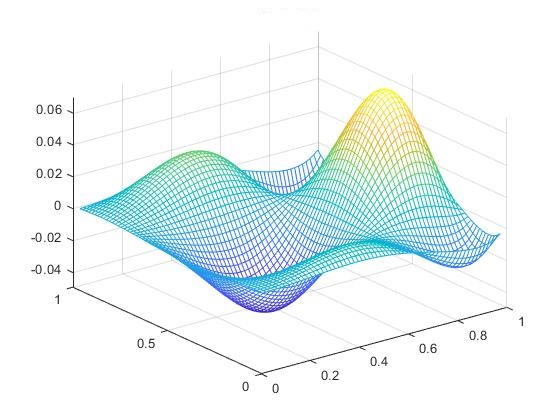}
	}
	\subfloat[$\alpha=1$]
	{
		\includegraphics[height=5cm, width=4.5cm]{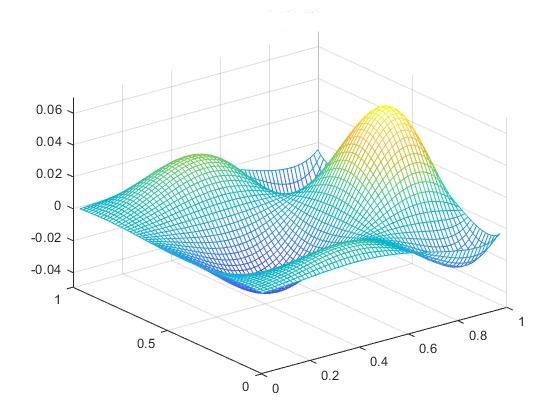}
	}
	%\end{figure}
	%\begin{figure}[h!]
	%\centering
	%\subfigure
	%{
	%\includegraphics[height=5cm, width=4.5cm]{epsM7.png}
	%}
	\caption{Solution of the M-matrix problem, $p_k \subset (0,1], p_k \to 0.1$, mesh size $h=\frac{1}{64}$. Results obtained by \textbf{Algorithm $1$} as modified in Remark \ref{rem:algphilog}.}
	\label{fig:MMatrixphi}
\end{figure}

\subsubsection{Time dependent control problem}\label{subsec:control}
The following example, of particular significance for our findings, is taken from \cite{GK1}, subsection $3.2$, to which we refer for further details. For the sake of completeness, we explain in the sequel the setting that we study.
We consider the linear control system
$$
\frac{d}{dt} z(t)=\mathcal{A} z(t)+D u(t), \quad z(0)=0,
$$
that is,
\begin{equation}\label{LCSfinalstate}
z(T)=\int_0^T e^{\mathcal{A}(T-s)} D u(s) ds,
\end{equation}
where the linear closed operator $\mathcal{A}$ generates a $C_0$-semigroup $e^{\mathcal{A}t}$, $t\geq 0$ on the state space $X$. More specifically, we consider the  one-dimensional controlled heat equation for $z=z(t,x)$:
\begin{equation}\label{actionscontrol}
z_t=z_{xx}+d_1(x)u_1(t)+d_2(x)u_2(t), \quad x \in (0,1),
\end{equation}
with homogeneous boundary conditions $z(t,0)=z(t,1)=0$ and thus $X=L^2(0,1)$. The differential operator $\mathcal{A}z=z_{xx}$ is discretized in space by the second order finite difference approximation with $n=49$ interior spatial nodes ($\Delta x=\frac{1}{50}$). We use two time dependent controls $\overrightarrow u=(u_1,u_2)$ with corresponding spatial control distributions $d_i$ chosen as step functions:
$$
d_1(x)=\chi_{(.2,.3)}, \quad d_2(x)=\chi_{(.6,.7)}.
$$
The control problem consists in finding the control function $\overrightarrow u$ that steers the state $z(0)=0$ to a neighbourhood of the desired state $z_T$ at the terminal time $T=1$. We discretize the problem in time by the mid-point rule and we define the matrix $A$ as follows
\begin{equation}\label{midpoint}
A \overrightarrow  u=\sum_{k=1}^m e^{\mathcal{A}\left(T-t_{k}-\frac{\Delta t}{2}\right)} (D \overrightarrow u)_k \Delta t,
\end{equation}
where $\overrightarrow u=(u_1^1,\cdots, u_1^m,u_2^1,\cdots u_2^m)$ is a discretized control vector whose coordinates represent the values at the mid-point of the intervals $(t_k,t_{k+1})$. Note that in \eqref{midpoint} we denote by $D$ a suitable rearrangement of the matrix $D$ in \eqref{LCSfinalstate} with some abuse of notation. A uniform step-size $\Delta t=\frac{1}{50}$ ($m=50$) is utilized. More specifically, we apply our scheme to the discretized optimal control problem in time and space where $x$ from \eqref{optprob2m} is the discretized control vector $u \in \R^{2m}$ which is mapped by $A$ to the discretized output $z$ at time $1$ by means of \eqref{midpoint} and
$y$ in \eqref{optprob2m} is the discretized target function chosen as the Gaussian distribution $z_T(x)=0.4\,\mbox{exp}(-70(x-.7)^2))$ centered at $x=.7$.
The parameter $\eps$ was initialized with $10^{-3}$ and decreased down to $10^{-8}$.
Since the second control distribution is well within the support of the desired state $z_T$,  the authority of this control is expected  to be stronger than that of the first one, which is away from the target.

The structure of the example allows us to exploit the flexibility of the $(p_k)$ sequences approach since we know the two regions where the distributions of the control are active.   Therefore we can "relax" the sparsity constraint by choosing sequences $(p_k)$ which are closer to $1$ in those regions, whereas in the areas where the solution is expected to be zero, we let the sequence tends towards the value used in \cite{GK1} in the same situation.

More specifically, in Tables \ref{tablecontrol1} and \ref{tablecontrol2} we report the results of our tests for  $p=0.5$
and  $(p_k)$ approximated by a vector of $N=100$ points  with  $p_0=0.51$ generated by the MATLAB command $(p_k)=flip(0.5+1./P)$ with $P=linspace(2, 100, N)$, respectively. In both tables, $\alpha$ is incrementally increasing by factor of $10$ from $10^{-2}$ to $1$. We report only the values for the second control $u_2$ since the first control $u_1$ is always zero. Note that the residue is always of the order of $10^{-15}$.
Moreover,  the quantity $|u_2|_p^p$ decreases for $\alpha$ increasing, as expected. In Table \ref{tablecontrol1}  we see that $|u_2|_0$ is always zero except that for $\alpha=10^{-2}$, whereas in Table \ref{tablecontrol2} for $\alpha=10^{-1}$ it is different from zero.  We remark , for smaller values of $\alpha$, we find values of $|u_2|_0$ different from zero (consistent with our expectation) also whit the $(p_k)$ sequence. However, here we choose these values of $\alpha$ since we  focus on the comparison between $(p_k)$  and $p$ fixed and for these values of $\alpha$ we found the most interesting results. Consistently with our expectations, for each value of $\alpha$  the number of iterations  and the residue are always smaller with the $(p_k)$ sequence  than for $p=0.5$ fixed. As an example, we refer to the first rows of Table \ref{tablecontrol1} and Table \ref{tablecontrol2} and we remark that the same holds true for different values of $\alpha$ than those of the above mentioned tables. As mentioned above, we believe that this improvement, as compared  to  the solution given by the $p$ fixed approach, resides in the fact that with $(p_k)$ sequences we are able to give a "different weight" to each component of the solution, thus simplifying the performance of the algorithm. 

Moreover, we found some interesting improvements even with respect to the level of sparsity in the solution. In this regard, we underline that both $|u_2|_0^c$ and $|u_2|^{p_k}$ are smaller for $(p_k)$ (see second and third rows of Table \ref{tablecontrol1}) than when $p=0.5$ is fixed ( see Table \ref{tablecontrol2}).  

In the sixth row of the tables we show the number of singular components of the vector $u_2$ at the end of the $\eps$-path following scheme, that is, $Sp:=\#\{ i\,\, |\,\, |(u)_2)_i|<\eps\}$. For all  considered values of $\alpha$, we have that $Sp$ is the same as  $|u_2|_0^c$, thus confirming the effectiveness of  the $\eps$-strategy.

We tested the algorithm when $\phi_k(t)=\log(t^{p_k}+1)$, according to the modification shown in Remark \ref{rem:algphilog}. The results reported in Table \ref{tablecontrol3} show that the algorithm works well  enough  in this case, too. However, for the tested value of $(p_k)$, the performance is not as good as  in the case where $\phi_k(t)=t^{p_k}$. Indeed, the residue is always bigger, the sparsity level of the solution is comparable but slightly smaller, and the number of iterates is slightly higher.

Finally, we remark that, if we change the initialization \eqref{initmoneps}, the method converges to the same solution with no remarkable modifications in the number of iterations.

\begin{table}[tbhp]\caption{Sparsity in a time-dependent control problem, $ (p_k)$, mesh size $h=\frac{1}{50}$. Results obtained by \textbf{Algorithm $1$}.}\label{tablecontrol1}
% res=$10^{-8}$, $\eps=10^{-7}:10^{-8}$, zero=$10^{-12}$} \label{controltable1} 
	\centering
	\begin{tabular}{|l|c|c|c|c|}
		\hline
		{\bf $\alpha$ }  &$10^{-2}$&$10^{-1}$&$1$\\
		\hline
		no. of iterates &28&33&19\\
		\hline
		{\bf $|u_2|_0^c$ } &99&100&100 \\
		\hline
		{\bf $|u_2|^p_p$ }&6&$10^{-6}$&$10^{-7}$ \\
		\hline
		$\mbox{ Residue }$  &$10^{-17}$ &$10^{-17}$&$10^{-18}$\\
		\hline
		$\mbox{Sp}$ &99&100&100\\
		\hline
	\end{tabular}
%\label{tablecontrol1}
	%	\par
	%	\bigskip
\end{table}

\begin{table}[tbhp]\caption{Sparsity in a time-dependent control problem, $ p=0.5$, mesh size $h=\frac{1}{50}$. Results obtained by \textbf{Algorithm $1$}.}\label{tablecontrol2}
% res=$10^{-8}$, $\eps=10^{-7}:10^{-8}$, zero=$10^{-12}$} \label{controltable1} 
	\centering
	\begin{tabular}{|l|c|c|c|c|}
		\hline
		{\bf $\alpha$ }  &$10^{-2}$&$10^{-1}$&$1$\\
		\hline
		no. of iterates &35&60&23\\
		\hline
		{\bf $|u_2|_0^c$ } & 99&99&100 \\
		\hline
		{\bf $|u_2|^p_p$ }& 6.3&4.7&$10^{-6}$ \\
		\hline
		$\mbox{ Residue }$  &$10^{-17}$ &$10^{-11}$&$10^{-16}$\\
		\hline
		$\mbox{Sp}$ &99&100&100\\
		\hline
	\end{tabular}
	%\label{tablecontrol2}
	%	\par
	%	\bigskip
\end{table}

\begin{table}[tbhp]\caption{Sparsity in a time-dependent control problem, $ (p_k)$, mesh size $h=\frac{1}{50}$. Results obtained by \textbf{Algorithm $1$} modified as in Remark \ref{rem:algphilog}.}\label{tablecontrol3}
% res=$10^{-8}$, $\eps=10^{-7}:10^{-8}$, zero=$10^{-12}$} \label{controltable1} 
	\centering
	\begin{tabular}{|l|c|c|c|c|}
		\hline
		{\bf $\alpha$ }  &$10^{-2}$&$10^{-1}$&$1$\\
		\hline
		no. of iterates &51&51&50\\
		\hline
		{\bf $|u_2|_0^c$ } &97&98&98 \\
		\hline
		{\bf $|u_2|^p_p$ }&10&$1.9$&$0.7$ \\
		\hline
		$\mbox{ Residue }$  &$10^{-1}$ &$10^{-2}$&$10^{-2}$\\
		\hline
		$\mbox{Sp}$ &97&98&98\\
		\hline
	\end{tabular}
%\label{tablecontrol1}
	%	\par
	%	\bigskip
\end{table}

\section{Iteratively reweighted $\ell^1$-minimization}\label{1-section}

The analysis for the reweighted $\ell^2$-minimization   cannot be directly adapted to the  $\ell^1$-minimization case, because of the absolute value function singularity. Hence, the setting of this section is finite dimensional. The reader is referred also to   \cite{chen2014convergence}, where  the fixed exponent penalty case is dealt with.

 If the problem under consideration is %\marginpar{maybe $\phi_{k}(\abs{x_{k}}+\eps)$?}
\[
\min_{x}\frac{1}{2}\norm{Ax-y}^{2}+ \alpha\sum_{k=1}^N\phi_{k}(\abs{x_{k}})
\]
with $\phi_{k}\in CI$ for all $k$, then we can use the concave conjugates of the $\phi_{k}$ functions to formulate the problem as
\begin{equation}\label{alt1}
\min_{x}\inf_{s_{k}\geq 0}\Big[G(x,s) = \frac{1}{2}\norm{Ax-y}^{2}+\alpha\sum_{k=1}^N (s_{k}\abs{x_{k}} - \phi_{k}^{\oplus}(s_{k}))\Big].
\end{equation}
Since $\phi_{k}$ is concave, $G$ in this problem is convex in $x$ and also convex in all the  $s_{k}$ (jointly), but not convex in $x$ and $s_{k}$ jointly. However, we can perform alternating minimization:
\begin{enumerate}
\item Initialize with some $x^{0}$ and $s^{0}$, set $n=0$ and iterate
\item $\displaystyle{x^{n+1}\in\argmin_{x}G(x,s^{n})}$
\item $\displaystyle{s^{n+1}\in\argmin_{s\geq 0} G(x^{n+1},s)}$.
\end{enumerate}
The minimization for $s$ decouples into one-dimensional problems
$\displaystyle{\min_{s}(s\abs{x_{k}} - \phi_{k}^{\oplus}(s))}$ with optimality conditions
$0 \in\abs{x_{k}}+\partial  (-\phi_{k}^{\oplus})(s)=\abs{x_{k}}- \hat{\partial}\phi_{k}^{\oplus}(s)$. By
Theorem~\ref{thm:fenchel}, this holds if and only if
$s\in \hat{\partial}\phi_{k}(\abs{x_{k}})$. We obtain the following method consisting of a sequence of convex
$\ell^1$-minimization problems:

Initialize with some $x^{0}$ and $s^{0}$, set $n=0$ and iterate
\begin{align}\label{reweighted_conv_gen}
x^{n+1}\in\argmin_{x}\frac{1}{2} \norm{Ax-y}^2 + \alpha\sum_{k=1}^N s_k^n\abs{x_k}
\end{align}
with
\begin{equation}\label{eps}
s_{k}^{n+1}\in\hat{\partial}\phi_{k}(\abs{x_{k}^{n+1}}+\eps).
\end{equation}
Here, we allow for a shift $\eps\geq 0$, since one may run into problems when $\hat{\partial}\phi_{k}(t)$ grows unbounded near zero (but still include $\eps=0$ to be able to treat the general case).

Thus, we will work with a shifted version of the regularized problem, that is
\begin{equation}\label{shifted}
    F(x) = \frac{1}{2}\norm{Ax-y}^{2}+ \alpha\sum_{k=1}^N\phi_{k}(\abs{x_{k}}+\eps),\quad \eps\geq 0 .
  \end{equation}

We start the analysis of the method with a proposition on descent:
\begin{proposition}\label{objective_mon}
  Assume that the functions $\phi_k$ are in $\CI$ and differentiable on $(0,\infty)$. Let
  $(x^{n})$ be generated by the algorithm \eqref{reweighted_conv_gen}-\eqref{eps}.
  Then it holds
\begin{equation}\label{decrease}
F(x^{n+1})\leq F(x^n)-\frac{1}{2}\norm{Ax^{n+1}-Ax^n}^2-\alpha D(x^{n+1},x^n),\,\forall n\in\NN,
\end{equation}
where $D(x^{n+1},x^n)$ is the sum of Bregman distances  associated to the convex functions $-\phi_k$, that is
\[\displaystyle{\sum_{k=1}^N D_{-\phi_k}} (|x_k^{n+1}|+\eps,|x_k^n|+\eps) = \sum_{k=1}^{N}\phi_{k}(\abs{x_{k}^{n}}+\eps) - \phi_{k}(\abs{x_{k}^{n+1}}+\eps) - s_{k}^{n}(\abs{x_{k}^{n}}-\abs{x_{k}^{n+1}})
\]
for some $s_{k}^{n}\in\hat\partial\phi_k(|x_k^n|+\eps)$.

\end{proposition}

\begin{proof}
  Due to the optimality condition for $x^{n+1}$,  there exists $\xi_k^{n+1}\in \partial (|\cdot|)(x^{n+1}_k)$ such that
\begin{equation}\label{optim}
(A^*(Ax^{n+1}-y))_k+\alpha s_{k}^{n}\xi_k^{n+1}=0,\,\, 1\leq k\leq N,
\end{equation}
for some $s_{k}^{n}\in\hat\partial\phi_{k}(\abs{x_{k}^{n}}+\eps)$.
Based on this, one has
\begin{eqnarray}\label{long}
F(x^n)-F(x^{n+1}) &=& \frac{1}{2}\|Ax^n-Ax^{n+1}\|^2+\langle Ax^n-Ax^{n+1},Ax^{n+1}-y\rangle\\
&+& \alpha\sum_{k=1}^N(\phi_k(|x_k^n|+\eps)-\phi_k(|x_k^{n+1}|+\eps))\nonumber\\
&=&\frac{1}{2} \|Ax^n-Ax^{n+1}\|^2+\alpha\sum_{k=1}^N (x_k^{n+1}-x_k^n)s_{k}^{n}\xi_k^{n+1}\nonumber\\
&+& \alpha\sum_{k=1}^N(\varphi_k(|x_k^n|+\eps)-\varphi_k(|x_k^{n+1}|+\eps))\nonumber\\
&\geq & \frac{1}{2}\|Ax^n-Ax^{n+1}\|^2+\alpha\sum_{k=1}^N (|x_k^{n+1}|-|x_k^n|)s_{k}^{n}\nonumber\\
&+& \alpha\sum_{k=1}^N(\varphi_k(|x_k^n|+\eps)-\varphi_k(|x_k^{n+1}|+\eps))\nonumber\\
&=&  \frac{1}{2}\|Ax^n-Ax^{n+1}\|^2+\alpha\sum_{k=1}^N D_{-\varphi_k} (|x_k^{n+1}|+\eps,|x_k^n|+\eps).\nonumber
\end{eqnarray}
Note that the last but one inequality followed from properties of subgradients of $|\cdot|$:  $\xi_k^{n+1}x_k^n\leq |x_k^n|$ and $\xi_k^{n+1}x_k^{n+1}=|x_k^{n+1}|$.
\end{proof}

%\begin{remark}
%Note that the functions $\varphi_i$ are differentiable on $(0,\infty)$, while of course the functions $t\mapsto \varphi_i(|t|+\eps)$ are not differentiable, but inherit the  Lipschitz property from  $\varphi_i$.
%\end{remark}

%Additional assumption: $\varphi_i$ should be coercive.

The next result shows convergence properties of the sequences generated by \eqref{reweighted_conv_gen}-\eqref{eps}. We recall first a notion that will be used in the analysis, namely sequential consistency of functions. A convex function $g$  defined on a Banach spaces  $X$ is called sequentially consistent on its domain if for any two sequences $(x^k)$ and $(y^k)$ such that the first is bounded, one has
\begin{equation}\label{consistent}
D_g(y^k,x^k)\to 0\quad\Rightarrow \|x^k-y^k\|\to 0,\quad\mbox{as}\,\,k\to\infty.
\end{equation}
There are several equivalent definitions for sequential consistency - see, e.g. Theorem 2.10 in \cite{but-res}. Here we will take advantage of  the  following: 

Let $g:\R^N\to(-\infty,\infty]$ be convex,  lower semicontinuous, Fr\'{e}chet differentiable
on its domain, such that  its derivative is uniformly continuous on bounded sets. Then $g$ is sequentially consistent.

We also recall   the notion of stationary point for a locally Lipschitz functional. According to \cite[Chapter 2]{clarke}, if $f:X\to\R$ is a locally Lipschitz function defined on a Banach space $X$, the generalized subdifferential of $f$ at $x\in X$ in the sense of Clarke is defined by
$$\partial f(x)=\{\xi\in X^*: f^\circ(x,v)\geq \langle\xi,v\rangle,\,\forall v\in X\},$$
where $f^\circ(x,v)$ is the Clarke directional derivative
$$f^\circ(x,v)=\limsup_{\substack{y\to x\\ t\searrow 0}}\frac{f(y+tv)-f(y)}{t}\,.$$
A point $x\in X$ is called  a stationary point of $f$, if $0\in\partial f(x).$ As expected, any local minimum point of $f$ is also a stationary point of the function.

In the remaining part of this section, we assume that the functions $t\to \phi_k(t+\eps)$  are  Lipschitz  for any $\eps>0$ (i.e., $\phi_k$ are uniformly Lipschitz on sets away from zero).% That is, there exists $L>0$, such that
%$$|\phi_k(t+\eps)-\phi_k(s+\eps)|\leq L|t-s|,\,\forall t,s>0,\,\forall k\in\NN.$$ 
Consequently, the mapping $\displaystyle \tilde \phi(x)= \sum_{k=1}^N\phi_{k}(\abs{x_{k}}+\eps)$ is also Lipschitz. Since the least squares term of $F$ is continuously differentiable and the regularizer is Lipschitz, it follows that the Clarke subdifferential of $F$ is additive, that is $\partial F(x)=A^*(Ax-y)+\partial \tilde \phi(x)$, where the latter Clarke sudifferential  is given by $\partial\tilde \phi(x)=((\phi_k'(|x_k|+\eps)\xi_k)_k$ with $\xi_k\in \partial (|\cdot|)(x_k)$.

%Moreover, we  recall a few notions of nonsmooth optimization. Note that  $F$ is neither differentiable (due to the absolute value function) nor convex. However, it is Lipschitz continuous when the $\phi_k$ functions are differentiable on $(0,\infty)$, as the mappings $t\mapsto \phi_k(t+\epsilon)$ have bounded derivatives on $(0, \infty)$.
%$$|\varphi(x)-\varphi(y)|\leq L\|x-y\|_1$$ for all $x,y$ and for some $L>0$. One can see this from
%$$|\varphi(x)-\varphi(y)|\leq  \sum_{i=1}^n |\varphi_i(|x_i|+\epsilon)-\varphi_i(|y_i|+\epsilon)|\leq  \sum_{i=1}^n L_i ||x_i|-|y_i||\leq L \|x-y\|_1,$$
%where $L=\max_i{L_i}$.

%As a Lipschitz function, $F$   is almost everywhere differentiable (see, e.g. ref Rockafellar) with the following  generalized subdifferential
%\begin{equation}
%\partial \varphi(x)=conv\{\lim_{l\to\infty}\nabla \varphi(x_l): x_l\to x\},
%\end{equation}
%with the notation 'conv' for the convex hull. A point $x\in{\RR}^n$ is called a stationary point of $\varphi$ if $0\in\partial \varphi(x)$. An optimality condition holds also in this framework, in the sense that a minimizer of $\varphi$ is necessarily a stationary point (reference).

\begin{proposition} Let $\varepsilon>0$. Let the functions $\phi_k$ be in $CI$, coercive and continuously differentiable. Assume that the functions $-\phi_k$ are sequentially consistent on domains that are bounded and away from zero. Then each sequence $(x^n)$ defined by \eqref{reweighted_conv_gen} is bounded and verifies $\displaystyle{\lim_{n\to\infty}(x^{n+1}-x^n)=0}$. Moreover, each accumulation point of $(x^n)$ is a stationary point of the function $F$.
\end{proposition}

\begin{proof} The sequence $(F(x^n))$ is convergent, as it is decreasing due to \eqref{decrease}, and is bounded from below. Thus, it 
 yields
\begin{equation}\label{zero}
\lim_{n\to\infty} D(x^{n+1},x^n)=0,\quad  \lim_{n\to\infty}\|Ax^n-Ax^{n+1}\|=0.
\end{equation}
Consequently,
 \begin{equation}\label{zero_i}
\lim_{n\to\infty} D_{-\phi_k} (|x_k^{n+1}|+\eps,|x_k^n|+\eps)=0,\quad \forall k\in\NN,
\end{equation}
which yields 
\begin{equation}\label{limit}
\displaystyle{\lim_{n\to\infty}\left(|x_k^{n+1}|-|x_k^n|\right)=0}
\end{equation}
 for each $k$, due to the sequential consistency of $-\phi_k$ for any $k$ on sets that are bounded and away from zero. Here we used  the boundedness of $(x^n)$  (that is implied by the boundedness of $(F(x^n))$) and an argument as in the proof of Lemma \ref{inclusion},
 \[
\frac{c}{1+L}\|x^n\|_1\leq \sum_{k=1}^N \frac{c|x_k^n|}{1+|x_k^n|}\leq c\sum_{k=1}^N \frac{|x_k^n|+\eps}{1+|x_k^n|+\eps}\leq \sum_{k=1}^N  \phi_k(|x_k^n|+\eps)\leq F(x^n).
\]
Now \eqref{limit} and the Lipschitz continuity of $\displaystyle x\mapsto \sum_{k=1}^N\phi_k(|x_k|+\eps)$  combined with the second equality in \eqref{long} lead to
\begin{equation}\label{interm}
\lim_{n\to\infty}\left(x_k^{n+1}-x_k^n\right)\xi_k^{n+1}=0
\end{equation}
 for each $k$.

Fix $k\in\{1,\dots,N\}$. One distinguishes two cases. If $x_k^{n+1}\neq 0$ except for finitely many $n\in\NN$, then $|\xi_k^{n+1}|=1$, which due to \eqref{interm} implies that 
\begin{equation}\label{interm1}
\lim_{n\to\infty}\left(x_k^{n+1}-x_k^n\right)=0.
\end{equation}
In the other case,  there is a subsequence $(x_k^{n_j+1})$ which is constant zero. Since

$\displaystyle{\lim_{n\to\infty}\left(|x_k^{n+1}|-|x_k^n|\right)=0}$, one has $x_k^{n_j}\to 0$  and hence, $\displaystyle{\lim_{j\to\infty}(x_k^{n_j+1}-x_k^{n_j})=0}$.

Let $(x^{n_j})$ be a subsequence which converges to $\bar x$ and fix  $k\in\{1,\dots,N\}$. By keeping the same notation on a subsequence of the bounded sequence $(\xi^{n_j}_k)$, the closeness of the graph of the subdifferential of $|\cdot|$ (see, e.g. Theorem 24.4 in \cite{rockafellar70})  yields existence of $\xi_k\in\partial (|\cdot|)(\bar x_k)$, such that $\xi^{n_j}_k\to\xi_k$ when $j\to\infty$. Due to \eqref{optim}, it follows that 
\begin{equation}\label{optim1}
(A^*(Ax^{n_j}-y))_k+\alpha \phi_k'(|x_k^{n_j-1}-x_k^{n_j}+x_k^{n_j}|+\eps)\xi_k^{n_j}=0,\,1\leq k\leq N.
\end{equation}
By taking limit for $j\to\infty$ in \eqref{optim1} and using the continuity of $\phi'_k$ imply
that $\bar x$ is a {stationary point of $F$}.
\end{proof}

\begin{remark} The functions from Remark \ref{examples}  satisfy the sequential consistency property on sets that are bounded and away from zero, since the component functions are  twice differentiable with bounded second derivatives, and thus, their derivatives are Lipschitz on sets that are bounded and away from zero. %Moreover, they are uniformly Lipschitz on sets away from zero.

\end{remark}

\section{Conclusions}
We proposed a  general nonconvex approach to reconstruct sparse solutions of ill-posed problems with the aim of enhancing more flexibility than in the classical regularization penalties in $\ell^p$ spaces with $p \in (0,\infty)$. Our analysis touched  both theoretical and 
numerical aspects. From a theoretical point of view,  we studied the convergence of the regularization method and  the convergence properties of a couple of majorization techniques, leading 
to monotone $\ell^2$-minimization and $\ell^1$-minimization schemes.
Finally, we  showed convergence of an algorithm based on $\ell^2$-minimization and tested it in two situations, first for an academic example where the operator is an M matrix, and secondly for a
time-dependent optimal control problem. As shown by numerical results, the procedure is  efficient and accurate, highlighting  the advantages of using
variable penalties over a fixed penalty.

Further work will concern a different numerical approach, that is, a primal dual active set method for the nonconvex problem. This method is inspired by the one proposed in \cite{GK1} to study nonconvex problems with an $\ell^p$ penalty, $p \in (0,1]$.  From a theoretical point of view, the primal dual active set method is interesting in itself since it identifies global minimizers (see Remark 5 of \cite{GK1}). Moreover, the method seems to be more numerically effective than the monotone algorithm in some situations. Last but not least, learning the flexible regularizers  is a challenging open problem. %Convergence of the method and the numerical performance in diverse situations will be studied.

\section{Acknowledgments}
The authors are grateful to Robert Csetnek (University of Vienna) for the discussions on generalized subdifferentials, and to the referee for the interesting comments and suggestions. The work of D.L. has been supported by the ITN-ETN project TraDE-OPT funded by the European Union’s Horizon 2020 research and innovation programme under the Marie Skłodowska-Curie grant agreement No 861137. This work represents only the author’s view and the European Commission is not responsible for any use that may be made of the information it contains. Furthermore, D.L. acknowledges funding from BMBF under grant 05M20MBB and from the DFG under grants LO 1436/9-1 and FOR 3022/1. 
 
\bibliographystyle{plain}
\bibliography{references}

\begin{thebibliography}{10}

\bibitem{akguen_yildirir}
R.~Akg\"un and Y.E. Yildirir.
\newblock Convolution and {J}ackson inequalities in {M}usielak--{O}rlicz
  spaces.
\newblock {\em Turkish Journal of Mathematics}, 42:2166--2185, 2018.

\bibitem{Allain2006}
Marc Allain, J{\'e}r{\^o}me Idier, and Yves Goussard.
\newblock On global and local convergence of half-quadratic algorithms.
\newblock {\em IEEE Transactions on Image Processing}, 15(5):1130--1142, 2006.

\bibitem{A}
H\'{e}dy Attouch, J\'{e}r\^{o}me Bolte, Patrick Redont, and Antoine Soubeyran.
\newblock Proximal alternating minimization and projection methods for
  nonconvex problems: an approach based on the {K}urdyka-\l ojasiewicz
  inequality.
\newblock {\em Math. Oper. Res.}, 35(2):438--457, 2010.

\bibitem{BeHa19}
Amir Beck and Nadav Hallak.
\newblock Optimization problems involving group sparsity terms.
\newblock {\em Mathematical Programming Series A}, 178:39--67, 2019.

\bibitem{Zhao12reweighted-l1-minimization}
Yun bin Zhao and Duan Li.
\newblock Reweighted `1-minimization for sparse solutions to underdetermined
  linear systems.
\newblock {\em SIAM Journal on Optimization}, pages 1065--1088, 2012.

\bibitem{5}
Michael~J. Black and Anand Rangarajan.
\newblock On the unification of line processes, outlier rejection, and robust
  statistics with applications in early vision.
\newblock {\em Int. J. Comput. Vis.}, 19:57--91, 1996.

\bibitem{Black1996}
Michael~J Black and Anand Rangarajan.
\newblock On the unification of line processes, outlier rejection, and robust
  statistics with applications in early vision.
\newblock {\em International journal of computer vision}, 19(1):57--91, 1996.

\bibitem{KLellp}
J\'{e}r\^{o}me Bolte, Shoham Sabach, and Marc Teboulle.
\newblock Proximal alternating linearized minimization for nonconvex and
  nonsmooth problems.
\newblock {\em Math. Program.}, 146(1-2, Ser. A):459--494, 2014.

\bibitem{superdiff}
K.C. Border.
\newblock Supergradients - lecture notes, 2001.
\newblock available at
  http://www.its.caltech.edu/$\sim$kcborder/Notes/Supergrad.pdf.

\bibitem{lorenz2008harditer}
Kristian Bredies and Dirk~A. Lorenz.
\newblock Iterated hard shrinkage for minimization problems with sparsity
  constraints.
\newblock {\em SIAM Journal on Scientific Computing}, 30(2):657--683, 2008.
\newblock [\href{http://dx.doi.org/10.1137/060663556}{doi}].

\bibitem{lorenz2008conv_speed_sparsity}
Kristian Bredies and Dirk~A. Lorenz.
\newblock On the convergence speed of iterative methods for linear inverse
  problems with sparsity constraints.
\newblock {\em Journal of Physics: Conference Series}, 124:012031 (12pp),
  September 2008.

\bibitem{bredies2009nonconvexregularization}
Kristian Bredies and Dirk~A. Lorenz.
\newblock Regularization with non-convex separable constraints.
\newblock {\em Inverse Problems}, 25(8):085011 (14pp), 2009.

\bibitem{bredies2014nonconvexminimization}
Kristian Bredies, Dirk~A. Lorenz, and Stefan Reiterer.
\newblock Minimization of non-smooth, non-convex functionals by iterative
  thresholding.
\newblock {\em Journal of Optimization Theory and Applications},
  165(1):78--112, 2015.

\bibitem{6}
Kristian Bredies, Dirk~A. Lorenz, and Stefan Reiterer.
\newblock Minimization of non-smooth, nonconvex functionals by iterative
  thresholding.
\newblock {\em J. Optim. Theory Appl.}, 165(1):78--112, 2015.

\bibitem{briceno2011proximal}
Luis~M Briceno-Arias, Patrick~L Combettes, J-C Pesquet, and Nelly Pustelnik.
\newblock Proximal algorithms for multicomponent image recovery problems.
\newblock {\em Journal of Mathematical Imaging and Vision}, 41(1):3--22, 2011.

\bibitem{but-res}
Dan Butnariu and Elena Resmerita.
\newblock Bregman distances, totally convex functions, and a method for solving
  operator equations in banach spaces.
\newblock {\em Abstract and Applied Analysis}, 2006:084919, 2006.

\bibitem{chaari2009minimization}
Lotfi Cha{\^a}ri, Jean-Christophe Pesquet, Amel Benazza-Benyahia, and Philippe
  Ciuciu.
\newblock Minimization of a sparsity promoting criterion for the recovery of
  complex-valued signals.
\newblock In {\em SPARS'09-Signal Processing with Adaptive Sparse Structured
  Representations}, 2009.

\bibitem{Chan2014}
Raymond~H Chan and Hai-Xia Liang.
\newblock Half-quadratic algorithm for ell \_p $-$ell \_q problems with
  applications to tv-$\ell_1$ image restoration and compressive sensing.
\newblock In {\em Efficient algorithms for global optimization methods in
  computer vision}, pages 78--103. Springer, 2014.

\bibitem{10}
Rick Chartrand.
\newblock Exact reconstruction of sparse signals via noconvex minimization.
\newblock {\em IEEE Signal Processing Letters}, 14(10):707--710, 2007.

\bibitem{11}
Rick Chartrand.
\newblock Fast algorithms for nonconvex compressive sensing: {MRI}
  reconstruction from very few data.
\newblock {\em 2009 IEEE International Symposium on Biomedical Imaging: From
  Nano to Macro}, pages 262--265, 2009.

\bibitem{chen1998basispursuit}
Scott~Shaobing Chen, David~L. Donoho, and Michael~A. Saunders.
\newblock Atomic decomposition by basis pursuit.
\newblock {\em SIAM Journal on Scientific Computing}, 20(1):33--61, 1998.

\bibitem{chen2014convergence}
Xiaojun Chen and Weijun Zhou.
\newblock Convergence of the reweighted $\ell_1$ minimization algorithm for
  $\ell_2$-$\ell_p$ minimization.
\newblock {\em Computational Optimization and Applications}, 59(1-2):47--61,
  2014.

\bibitem{claerbout1973robust}
Jon~F. Claerbout and Francis Muir.
\newblock Robust modeling with erratic data.
\newblock {\em Geophysics}, 38(5):826--844, 1973.

\bibitem{clarke}
F.H. Clarke.
\newblock {\em Optimization and Nonsmooth Analysis}.
\newblock Wiley New York, 1983.

\bibitem{daubechies2004iteratethresh}
Ingrid Daubechies, Michel Defrise, and Christine {De Mol}.
\newblock An iterative thresholding algorithm for linear inverse problems with
  a sparsity constraint.
\newblock {\em Communications in Pure and Applied Mathematics},
  57(11):1413--1457, 2004.

\bibitem{daubechies2003iteratethresh}
Ingrid Daubechies, Michel Defrise, and Christine {De Mol}.
\newblock An iterative thresholding algorithm for linear inverse problems with
  a sparsity constraint.
\newblock {\em Communications in Pure and Applied Mathematics},
  57(11):1413--1457, 2004.

\bibitem{SCAD}
J.~Fan.
\newblock Comments on “wavelets in statistics: A review by a. antoniadis.”.
\newblock {\em J. Ital. Stat. Soc.}, 6:131--138, 1997.

\bibitem{figueiredo2007gradproj}
M{\'a}rio A.~T. Figueiredo, Robert~D. Nowak, and Stephen~J. Wright.
\newblock Gradient projection for sparse reconstruction: Applications to
  compressed sensing and other inverse problems.
\newblock {\em IEEE Journal of Selected Topics in Signal Processing},
  4:586--597, 2007.

\bibitem{figueiredo2009sparsa}
M{\'a}rio A.~T. Figueiredo, Robert~D. Nowak, and Stephen~J. Wright.
\newblock Sparse reconstruction by separable approximation.
\newblock {\em IEEE Transactions on Signal Processing}, 57:2479--2493, 2009.

\bibitem{Geman1992}
D.~{Geman} and G.~{Reynolds}.
\newblock Constrained restoration and the recovery of discontinuities.
\newblock {\em IEEE Transactions on Pattern Analysis and Machine Intelligence},
  14(3):367--383, 1992.

\bibitem{Geman1995}
Donald Geman and Chengda Yang.
\newblock Nonlinear image recovery with half-quadratic regularization.
\newblock {\em IEEE Transactions on Image Processing}, 4(7):932--946, 1995.

\bibitem{GK1}
D.~Ghilli and K.~Kunisch.
\newblock On monotone and primal dual active set schemes for sparsity
  optimization in $\ell^p$ with $p \in (0, 1)$.
\newblock {\em Comput Optim Appl.}, 72(1):45--85, 2018.

\bibitem{Gorodnitsky1997}
Irina~F Gorodnitsky and Bhaskar~D Rao.
\newblock Sparse signal reconstruction from limited data using focuss: A
  re-weighted minimum norm algorithm.
\newblock {\em IEEE Transactions on signal processing}, 45(3):600--616, 1997.

\bibitem{grasmair2008pleq1}
Markus Grasmair.
\newblock Well-posedness and convergence rates for sparse regularization with
  sublinear $l^q$ penalty term.
\newblock {\em Inverse Problems in Imaging}, 3(3):383--387, 2009.

\bibitem{grasmair2010nonconvexsparse}
Markus Grasmair.
\newblock Non-convex sparse regularization.
\newblock {\em Journal of Mathematical Analysis and Applications}, 365:19--28,
  2010.

\bibitem{grasmair2011homogeneous}
Markus Grasmair.
\newblock Linear convergence rates for tikhonov regularization with positively
  homogeneous functionals.
\newblock {\em Inverse Problems}, 27(7):075014, 2011.

\bibitem{grasmair2008sparseregularization}
Markus Grasmair, Markus Haltmeier, and Otmar Scherzer.
\newblock Sparse regularization with $\ell^q$ penalty term.
\newblock {\em Inverse Problems}, 24(5):055020 (13pp), 2008.

\bibitem{lorenz2008ssn}
Roland Griesse and Dirk~A. Lorenz.
\newblock A semismooth {N}ewton method for {T}ikhonov functionals with sparsity
  constraints.
\newblock {\em Inverse Problems}, 24(3):035007 (19pp), 2008.

\bibitem{HuLiMeQiYa17}
Yaohua Hu, Chong Li, Kaiwen Meng, Jing Qin, and Xiaoqi Yang.
\newblock Group sparse optimization via $\ell_{p,q}$ regularization.
\newblock {\em Journal of Machine Learning Research}, 18(30):1--52, 2017.

\bibitem{lai2013irls}
Ming-Jun Lai, Yangyang Xu, and Wotao Yin.
\newblock Improved iteratively reweighted least squares for unconstrained
  smoothed {$\ell_q$} minimization.
\newblock {\em SIAM J. Numer. Anal.}, 51(2):927--957, 2013.

\bibitem{Lanza2015}
Alessandro Lanza, Serena Morigi, Lothar Reichel, and Fiorella Sgallari.
\newblock A generalized krylov subspace method for
  $\backslash$ell\_p-$\backslash$ell\_q minimization.
\newblock {\em SIAM Journal on Scientific Computing}, 37(5):S30--S50, 2015.

\bibitem{levy1981reconstruction}
Shlomo Levy and Peter~K. Fullagar.
\newblock Reconstruction of a sparse spike train from a portion of its spectrum
  and application to high-resolution deconvolution.
\newblock {\em Geophysics}, 46(9):1235--1243, 1981.

\bibitem{36}
Guoyin Li and Ting~Kei Pong.
\newblock Global convergence of splitting methods for nonconvex composite
  optimization.
\newblock {\em SIAM J. Optim.}, 25(4):2434--2460, 2015.

\bibitem{lorenz2008reglp}
Dirk~A. Lorenz.
\newblock Convergence rates and source conditions for {T}ikhonov regularization
  with sparsity constraints.
\newblock {\em Journal of Inverse and Ill-Posed Problems}, 16(5):463--478,
  2008.

\bibitem{lorenz2016flexiblesparse}
Dirk~A. Lorenz and Elena Resmerita.
\newblock Flexible sparse regularization.
\newblock {\em Inverse Problems}, 33(1), 2016.

\bibitem{37}
Zhaosong Lu.
\newblock Iterative reweighted minimization methods for {$l_p$} regularized
  unconstrained nonlinear programming.
\newblock {\em Math. Program.}, 147(1-2, Ser. A):277--307, 2014.

\bibitem{Nikolova2005}
Mila Nikolova and Michael~K Ng.
\newblock Analysis of half-quadratic minimization methods for signal and image
  recovery.
\newblock {\em SIAM Journal on Scientific computing}, 27(3):937--966, 2005.

\bibitem{pustelnik2011parallel}
Nelly Pustelnik, Caroline Chaux, and Jean-Christophe Pesquet.
\newblock Parallel proximal algorithm for image restoration using hybrid
  regularization.
\newblock {\em IEEE transactions on Image Processing}, 20(9):2450--2462, 2011.

\bibitem{44}
Ronny Ramlau and Clemens~A. Zarzer.
\newblock On the minimization of a {T}ikhonov functional with a non-convex
  sparsity constraint.
\newblock {\em Electron. Trans. Numer. Anal.}, 39:476--507, 2012.

\bibitem{rao1997focuss}
Bhaskar~D. Rao and Kenneth Kreutz-Delgado.
\newblock Deriving algorithms for computing sparse solutions to linear inverse
  problems.
\newblock In {\em Asilomar Conference on Signals, Systems and Computers},
  volume~1, pages 955--959, Monterey, California, November 1997.

\bibitem{rockafellar70}
R.~Tyrell Rockafellar.
\newblock {\em Convex Analysis}.
\newblock Princeton Univ. Press, 1970.

\bibitem{ramlau2010sparse}
Ramlau Ronny and Elena Resmerita.
\newblock Convergence rates for regularization with sparsity constraints.
\newblock {\em Electronic Transactions on Numerical Analysis}, 37:87--104,
  2010.

\bibitem{santosa1986linear}
Fadil Santosa and William~W. Symes.
\newblock Linear inversion of band-limited reflection seismograms.
\newblock {\em SIAM Journal on Scientific and Statistical Computing},
  7(4):1307--1330, 1986.

\bibitem{taylor1979deconvolution}
Howard~L Taylor, Stephen~C Banks, and John~F McCoy.
\newblock Deconvolution with the $\ell_1$ norm.
\newblock {\em Geophysics}, 44(1):39--52, 1979.

\bibitem{tibshirani1996lasso}
Robert Tibshirani.
\newblock Regression shrinkage and selection via the lasso.
\newblock {\em Journal of the Royal Statistical Society. Series B},
  58(1):267--288, 1996.

\bibitem{vandenberg2007spgl1}
Ewout van~den Berg and Michael~P. Friedlander.
\newblock {SPGL1}: A solver for large-scale sparse reconstruction, June 2007.
\newblock http://www.cs.ubc.ca/labs/scl/spgl1.

\bibitem{wipf2010irls}
D.~{Wipf} and S.~{Nagarajan}.
\newblock Iterative reweighted $\ell_1$ and $\ell_2$ methods for finding sparse
  solutions.
\newblock {\em IEEE Journal of Selected Topics in Signal Processing},
  4(2):317--329, April 2010.

\bibitem{XuChXuZh12}
Zongben Xu, Xiangyu Chang, Fengmin Xu, and Hai Zhang.
\newblock $l_{1/2}$ regularization: A thresholding representation theory and a
  fast solver.
\newblock {\em IEEE Transactions on Neural Networks and Learning Systems},
  23(7):1013--1027, 2012.

\bibitem{yu2019iteratively}
Peiran Yu and Ting~Kei Pong.
\newblock Iteratively reweighted $\ell_1$ algorithms with extrapolation.
\newblock {\em Computational Optimization and Applications}, 73(2):353--386,
  2019.

\bibitem{zarzer2009nonconvextikhonov}
Clemens~A. Zarzer.
\newblock On {T}ikhonov regularization with non-convex sparsity constraints.
\newblock {\em Inverse Problems}, 25(2):025006, 2009.

\bibitem{MCP}
C.-H. Zhang.
\newblock Nearly unbiased variable selection under minimax concave penalty.
\newblock {\em Ann. Stat.}, 38:894--942, 2010.

\bibitem{52}
Wangmeng Zuo, Deyu Meng, Lei Zhang, Xiangchu Feng, and David Zhang.
\newblock A generalized iterated shrinkage algorithm for non-convex sparse
  coding.
\newblock {\em 2013 IEEE International Conference on Computer Vision}, pages
  217--224, 2013.

\end{thebibliography}

\end{document}